\documentclass[12pt]{amsart}
\usepackage[margin=1.25in]{geometry}
\usepackage{latexsym}
\usepackage{amsmath}
\usepackage{amssymb}
\usepackage{amsthm}
\usepackage{stmaryrd}
\usepackage{amscd}
\usepackage{enumerate} 
\usepackage{amssymb} 
\usepackage{mathrsfs}
\usepackage[all]{xy}
\usepackage{color}
\usepackage{graphicx}
\usepackage{mathtools}
\usepackage{comment}
\usepackage{multirow}
\usepackage{hyperref}
\hypersetup{colorlinks=true}
\usepackage{tikz}
\usepackage{tikz-cd}
\usetikzlibrary{decorations.markings}
\tikzset{degil/.style={
            decoration={markings,
            mark= at position 0.5 with {
                  \node[transform shape] (tempnode) {$\backslash$};
                  }
              },
              postaction={decorate}
}
}

\makeatletter
  
  \@addtoreset{equation}{thm}
  \makeatother

\title{On Frobenius liftability of surface singularities}
\author{Tatsuro Kawakami}
\address{Department of Mathematics, Graduate School of Science, Kyoto University, Kyoto 606-8502, Japan} 
\email{tatsurokawakami0@gmail.com}
\author{Teppei Takamatsu}
\address{Department of Mathematics (Hakubi Center), Graduate School of Science, Kyoto University, Kyoto 606-8502, Japan}
\email{teppeitakamatsu.math@gmail.com}

\def\phi{\varphi}
\def\epsilon{\varepsilon}
\def\tilde{\widetilde}

\def\mapsto{\longmapsto}

\def\log{\operatorname{log}}
\def\Hom{\operatorname{Hom}}
\def\Spec{\operatorname{Spec}}

\def\Supp{\operatorname{Supp}}

\def\Exc{\operatorname{Exc}}

\DeclareMathOperator\Tor{\mathrm{Tor}}

\newcommand{\Q}{\mathbb{Q}}

\newcommand{\Z}{\mathbb{Z}}
\newcommand{\PP}{\mathbb{P}}

\newcommand{\sO}{\mathcal{O}}
\newcommand{\sHom}{\mathop{\mathcal{H}\! \mathit{om}}}

\newcommand{\pushoutcorner}[1][ul]{\save*!/#1-1.8pc/#1:(-1,1)@^{|-}\restore}

\theoremstyle{plain}
\newtheorem{thm}{Theorem}[section] 
\newtheorem{cor}[thm]{Corollary}
\newtheorem{prop}[thm]{Proposition}
\newtheorem{conj}[thm]{Conjecture}

\newtheorem{lem}[thm]{Lemma}
\theoremstyle{definition} 
\newtheorem{defn}[thm]{Definition}
\newtheorem{conv}[thm]{Convention}

\theoremstyle{remark}
\newtheorem{rem}[thm]{Remark}

\newtheorem{defn and notation}[thm]{Definition and Notation}

\newtheorem{cln}{Claim}

\theoremstyle{plain}
\newtheorem{theo}{Theorem}


\begin{document}
\tolerance = 9999

\maketitle
\markboth{Tatsuro Kawakami and Teppei Takamatsu}{Frobenius liftability of surface singularities}

\begin{abstract}
We show that a plt surface singularity $(P\in X,B)$ is $F$-liftable if and only if it is $F$-pure and is not a 
rational double point of type $E_8^1$ in characteristic $p=5$.
As a consequence, we prove the logarithmic extension theorem for $F$-pure surface pairs and Bogomolov-Sommese vanishing for globally $F$-split surface pairs. These results were previously known to hold in characteristic $p>5$.
\end{abstract}

\tableofcontents

\section{Introduction}
We work over an algebraically closed field $k$ of characteristic $p>0$ unless stated otherwise.
A pair $(X,B)$ of a normal variety $X$ and a reduced divisor $B$ is said to be \textit{$F$-liftable} if there exists a lifting $(\tilde{X},\tilde{B})$ of $(X,B)$ over the ring $W_2(k)$ of Witt vectors of length two, and a lifting $\tilde{F}_X$ of Frobenius to $\tilde{X}$ that is compatible with $\tilde{B}$ (see Definition \ref{def:F-lift} for details).

Global $F$-liftability has been investigated by many authors and is known to have very good properties (see \cite{BTLM, AWZ, AWZ2} for example). 
Local $F$-liftability has been studied by Zdanowicz \cite{Zda18}; he established a criterion of $F$-liftability of hypersurfaces, and clarified the relationship between $F$-liftability and other classical $F$-singularities, such as $F$-pure and strongly $F$-regular singularities.
Recently, Langer \cite{Langer23, Langer24} proved Bogomolov’s inequality and Simpson’s correspondence for (globally) $W_2(k)$-liftable varieties with singularities related to $F$-liftability.
In addition, the first author \cite{Kaw4} proved the following extendability of differential forms for locally $F$-liftable pairs.

\begin{thm}[\textup{\cite[Theorem 4.2]{Kaw4}}]\label{thm:ext thm for F-lift}\label{thm:F-lift}
Let $X$ be a normal variety and $B$ a reduced divisor on $X$ such that $(X, B)$ is locally $F$-liftable.
Let $D$ be a $\Q$-Cartier $\Z$-divisor on $X$.
Then, for any proper birational morphism $f\colon Y\to X$ from a normal variety $Y$,
the restriction map
\[
f_{*}\Omega^{[i]}_Y(\log f_{*}^{-1}B+E)(f^{*}D)\hookrightarrow\Omega_X^{[i]}(\log B)(D)
\]
is an isomorphism for all $i\geq 0$, where $E$ is the reduced $f$-exceptional divisor.
We refer to Subsection \ref{subsection:Notation and terminology} for detailed notation.
\end{thm}

\subsection{$F$-liftability for surface singularities}
First, we investigate $F$-liftability of a normal surface singularity $(P\in X,B)$ such that $B$ is reduced.
If $(P\in X,B)$ is $F$-liftable, then it is $F$-pure (Corollary \ref{cor:F-lift to F-pure}), and $F$-pure surface singularities are classified in \cite{Mehta-Srinivas(surface)}, \cite{Hara(two-dim)}, and \cite{Hara-Watanabe}.
Therefore, it suffices to investigate which $F$-pure surface singularities are $F$-liftable.
In fact, we show that, for plt surface singularities, $F$-liftability coincides with $F$-purity except for a rational double point (RDP, for short) of type $E_8^1$ in $p=5$.

\begin{theo}\label{Introthm:F-purity and F-liftability}
    Let $k$ be an algebraically closed field of characteristic $p>0$.
    Let $(P\in X,B)$ be a plt surface singularity over $k$ such that $B$ is reduced.
    Then $(P\in X,B)$ is $F$-liftable if and only if it is $F$-pure and the following does not hold:
    \[
    p=5, B=0,\text{and } (P\in X) \text{\,\,is an RDP of type $E_8^1$}.
    \]
\end{theo}

It is natural to ask whether Theorem \ref{Introthm:F-purity and F-liftability} is still true without the plt assumption.
The difficulty lies in the $F$-liftability of a cusp singularity, which is a singularity whose exceptional divisor of the minimal resolution consists of a node or a cycle of rational curves.
In fact, if we assume that a cusp singularity is $F$-liftable (Conjecture \ref{conj}), then 
Theorem \ref{Introthm:F-purity and F-liftability} holds for all surface singularities (Proposition \ref{prop:F-pure to F-lift(non-plt)}).
A cusp singularity is $F$-pure \cite[Theorem 1.2]{Mehta-Srinivas(surface)}, and the authors expect that it is $F$-liftable. The $F$-liftability of some hypersurface cusp singularities in characteristic two will be confirmed in Section \ref{subsec:cusp}.

To deduce $F$-liftability of some of lc surface singularities, we use descent properties of $F$-liftability for finite covers.
To show this, we characterize $F$-liftability via a splitting of a reflexive version of the Cartier operator:

\textit{Let $(X,B)$ be a pair of a normal variety and a reduced divisor. Then $(X,B)$ is $F$-liftable if and only if the reflexive Cartier operator $C_{X,B}^{[1]}\colon Z_X^{[1]}(\log\,B)\to \Omega_X^{[1]}(\log\,B)$ splits.}

This characterization provides some new descent properties of $F$-liftability.
We refer to Section \ref{sec:F-lift and Car operator} for details.

\subsection{Applications of Theorem \ref{Introthm:F-purity and F-liftability}}
\subsubsection{Extendability of differential forms for $F$-pure surface pairs}

As an application of Theorem \ref{Introthm:F-purity and F-liftability}, we obtain the following theorem.

\begin{theo}\label{Introthm:LET}
Let $(X,B)$ be an $F$-pure surface pair over a perfect field of positive characteristic.
Let $D$ be a $\Z$-divisor on $X$.
Then, for any proper birational morphism $f\colon Y\to X$ from a normal variety $Y$, the restriction map
\[
f_{*}\Omega^{[i]}_Y(\log\,f^{-1}_{*}\lfloor B\rfloor+E)(f^{*}D)\hookrightarrow \Omega^{[i]}_X(\log\,\lfloor B\rfloor)(D)
\]
is an isomorphism for all $i\geq 0$, where $E$ is the reduced $f$-exceptional divisor.
\end{theo}
\begin{rem}\label{rem:intro}
The crucial statement is the case $D\neq 0$.
In fact, if $D=0$, then the above theorem is obtained by combining Graf's and Hirokado's \cite{Gra, Hirokado} results, or by \cite{Kaw6}.
However, deducing the case $D\neq 0$ from the case $D=0$ is not straightforward, even when $D$ is $\Q$-Cartier.
This is because the index one cover of $D$ can be inseparable in positive characteristic (see \cite[Remark 4.10]{Kaw3} for more details).
The case $D\neq 0$ is essential for the proof of Bogomolov-Sommese vanishing for globally $F$-split surface pairs (Theorem \ref{Introthm:BSV}).
\end{rem}
\begin{rem}
    Extendability, as stated in Theorem \ref{Introthm:LET}, holds for an lc surface pair over a perfect field of characteristic $p>5$. However, this fails when $p\leq 5$ (see \cite[Theorem 1.2]{Gra} and \cite[Proposition 4.8]{Kaw3}). 
\end{rem}

Using a technique of Graf \cite{Gra}, Theorem \ref{Introthm:LET} can be reduced to a dlt blow-up of $(X,B)$.
Then, except for the case where $p=5$ and the dlt blow-up has an RDP of type $E_8^1$, the extendability in the theorem is an immediate consequence of Theorem \ref{Introthm:F-purity and F-liftability} and Theorem \ref{thm:F-lift}.
In the remaining case, we obtain the desired extendability from the fact that the class group of the singularity is trivial.

\subsubsection{Bogomolov-Sommese vanishing for globally $F$-split surface pairs}

As another application of Theorem \ref{Introthm:F-purity and F-liftability}, we obtain the following theorem.

\begin{theo}\label{Introthm:BSV}
Let $(X,B)$ be a globally $F$-split projective surface pair over a perfect field of positive characteristic. 
Then 
\[
H^0(X, \Omega_X^{[i]}(\log\,\lfloor B\rfloor)(-D))=0
\]
for every $\Z$-divisor $D$ on $X$ satisfying $\kappa(X, D)>i$, where $\kappa(X,D)$ the Iitaka dimension of a $\Z$-divisor $D$ (see \cite[Definition 2.18]{GKKP}).
\end{theo}
\begin{rem}\,
\begin{enumerate}
    \item Theorem \ref{Introthm:BSV} holds for an lc projective surface pair $(X,B)$ over a perfect field of characteristic $p>5$ such that $\kappa(X, K_X+\lfloor B\rfloor)=-\infty$. However, this fails in $p\leq 5$ (see \cite[Theorem 1.2]{Kaw3}).
    \item If $p>5$, then Theorem \ref{Introthm:BSV} has already proven in \cite[Proposition 4.13]{Kaw3}.
\end{enumerate}
\end{rem}

The vanishing of Theorem \ref{Introthm:BSV} can be reduced to a dlt blow-up and a log resolution by Theorem \ref{Introthm:LET}.
Specifically, we crucially use the case where $D\neq 0$ of Theorem \ref{Introthm:LET}.
Then the combining the strategy in \cite{Kaw3} and \cite{BBKW}, we can deduce Theorem \ref{Introthm:BSV}.

\subsubsection{Another application}
By combining Theorem \ref{Introthm:F-purity and F-liftability} and Hirokado's result \cite{Hirokado}, we compare $F$-liftability, classical $F$-singularities, and extendability of differential forms for RDPs (see Table \ref{table:RDPs}).

\section{Preliminaries}

\subsection{Notation and terminology}\label{subsection:Notation and terminology}
Throughout the paper, we work over a fixed algebraically closed field $k$ of characteristic $p>0$ unless stated otherwise.
\begin{enumerate}
 \item A \textit{variety} means an integral separated scheme of finite type.
    \item For a proper birational morphism $f\colon Y\to X$ of schemes, we denote by $\Exc(f)$ the reduced $f$-exceptional divisor.
    \item A pair $(X,B)$ consists of a normal variety $X$ and an effective $\Q$-divisor $B$. A pointed pair $(x\in X,B)$ consists of a pair $(X,B) $ and a closed point $x\in X$. When $B=0$, we simply write $(x\in X,B)$ as $(x\in X)$.
    \item A morphism $g\colon (x'\in X',B')\to (x\in X,B)$ of pointed pairs is a morphism $g\colon X'\to X$ of normal varieties such that $g(B')\subset B$ and $g(x')=x$.
    \item A morphism $(u\in U,B_U)\to (x\in X,B)$ of pointed pairs is said to be
    an \textit{\'etale neighborhood of $x$} if $U\to X$ is \'etale and $B_U\cong B\times_X U$.
    \item When we identify two pointed pairs
    $(x_1\in X_1,B_1)$ and $(x_2\in X_2,B_2)$ with a common \'etale neighborhood, we call $(x_1\in X_1,B_1)$ (and $(x_2\in X_2,B_2)$) a \textit{singularity}.
    \item Given a normal variety $X$, we say that $U\subset X$ is a \textit{big open subscheme} if it is an open subscheme whose complement has codimension at least two.
    \item Given a pair $(X,B)$, we say $(X,B)$ is log smooth if $X$ is smooth and $B$ has simple normal crossing (snc, for short) support.
    \item Given a normal variety $X$, a reduced divisor $B$ on $X$, a $\Q$-divisor $D$ on $X$, and an integer $i\geq 0$, we denote $j_{*}(\Omega_U^{i}(\log B)\otimes \sO_U(\lfloor D\rfloor))$ by $\Omega_X^{[i]}(\log B)(D)$, where $j\colon U\hookrightarrow X$ is the inclusion of the log smooth locus $U$.
    \item Given a $\Z$-divisor $D$ on a normal variety $X$, the \textit{Iitaka dimension of $D$} is can be defined (see \cite[Definition 2.18]{GKKP}) and is denoted by $\kappa(X,D)$.
    \item For the definition of the singularities of pairs appearing in the MMP (such as \emph{canonical, klt, plt, lc}) we refer to \cite[Definition 2.8]{Kol13}. Note that we always assume that the boundary divisor is effective although \cite[Definition 2.8]{Kol13} does not impose this assumption.
\end{enumerate}

\subsection{\texorpdfstring{$F$}--singularities}
In this subsection, we summarize the facts about $F$-purity and $F$-liftability.
\begin{defn}\label{def:F-pure}
Let $(X,B)$ be a pair of a normal variety $X$, $B$ an effective $\Q$-divisor, and $P\in X$ a closed point.
\begin{enumerate}
    \item We say $(X,B)$ is \textit{$F$-pure at $P\in X$} if the map
\[
\sO_{X,P}\to F^e_{*}\sO_{X}((p^e-1)B)_{P}
\] 
splits as an $\sO_{X,P}$-module homomorphism for all $e>0$.
We say that $(X,B)$ is \textit{$F$-pure} if it is $F$-pure at every closed point $P\in X$.
\item We say that $(X,B)$ is \textit{globally $F$-split} if 
    \[
    \sO_X\to F^e_{*}\sO_X(\lfloor (p^e-1)B\rfloor) 
    \]
    splits as an $\sO_X$-module homomorphism for some $e>0$.
\item 
We say that $(X,B)$ is \textit{strongly $F$-regular at $P\in X$} if, for any non-zero element $c\in\sO_{X,P}^{\circ}$, there exists a positive integer $e>0$ such that
    \[
    \sO_{X,P}\to F^e_{*}\sO_{X,P}\xrightarrow{\times F^e_{*}c} F^e_{*}\sO_{X,P}
    \]
    splits as an $\sO_{X,P}$-module homomorphism.
    We say that $(X,B)$ is \textit{strongly $F$-regular} if it is strongly $F$-regular at every closed point $P\in X$.
\end{enumerate}
\end{defn}

\begin{defn}[F-liftability]\label{def:F-lift}
Let $X$ be a normal variety and $B=\sum_{r=1}^n B_r$ a reduced divisor on $X$, where every $B_r$ is an irreducible component.  
We denote the ring of Witt vectors of length two by $W_2(k)$.

We say that $(X,B)$ is \textit{$F$-liftable} 
if there exist 
\begin{itemize}
	\item a flat morphism $\widetilde{X} \to \Spec W_2(k)$ together with a closed immersion $i\colon X\hookrightarrow \widetilde{X}$, 
	 \item a closed subscheme $\widetilde{B}_r$ of $\widetilde{X}$ flat over $W_2(k)$ for all $r\in\{1,\ldots,n\}$, and 
	 \item a morphism $\widetilde{F}\colon \widetilde{X}\to \widetilde{X}$ over $W_2(k)$
\end{itemize}
	such that 
\begin{itemize}
	\item the induced morphism $i\times_{W_2(k)}k \colon X\to \widetilde{X}\times_{W_2(k)} k$ is an isomorphism, 
	\item $(i\times_{W_2(k)}k)(B_r)= \widetilde{B}_r\times_{W_2(k)} k$ for all $r\in\{1,\ldots,n\}$, and 
    \item $\widetilde{F}\circ i=i \circ F$, and $\widetilde{F}^{*}(\widetilde{B}_r|_U)=p(\widetilde{B}_r|_U)$ for all $r\in\{1,\ldots,n\}$, where $U$ is the log smooth locus of $(X,B)$.
\end{itemize}
In this case, we also say that $(\tilde{X}, \tilde{B}, \tilde{F})$ is a $W_2(k)$-lift of $(X,B,F)$.

We say that $(X,B)$ is \textit{$F$-liftable at a closed point $P\in X$} if there exists an open neighborhood $U$ of $P\in X$ such that $(U,B|_{U})$ is $F$-liftable.
We say that $(X,B)$ is \textit{locally $F$-liftable} if the pair $(X,B)$ is $F$-liftable at every closed point $P\in X$.
\end{defn}

\section{$F$-liftability and reflexive Cartier operators}\label{sec:F-lift and Car operator}

\subsection{Reflexive Cartier operators}
Throughout this subsection, we use the following convention.

\begin{conv}
Let $X$ be a normal variety and $B$ is a reduced divisor on $X$. 
Let $U$ be the log smooth locus of $(X,B)$ and $j\colon U\hookrightarrow X$ the inclusion.
\end{conv}

The Frobenius pushforward of the de Rham complex
\[
F_{*}\Omega^{\bullet}_U (\log\,B) \colon  F_{*}\sO_U(\log\,B) \xrightarrow{F_{*}d} F_{*} (\Omega^1_U\log\,B)  \xrightarrow{F_{*}d} \cdots
\]
is a complex of $\sO_U$-modules.

We define coherent $\sO_U$-modules by 
\begin{align*}
    &B^{i}_U(\log\,B) \coloneqq \mathrm{Im}(F_{*}d\otimes\sO_U \colon F_{*}\Omega^{i-1}_U(\log\,B)  \to F_{*}\Omega^{i}_U(\log\,B) ),\\
    &Z^{i}_U(\log\,B) \coloneqq \mathrm{Ker}(F_{*}d\otimes\sO_U \colon F_{*}\Omega^{i}_U(\log\,B)  \to F_{*}\Omega^{i+1}_U(\log\,B) ),
\end{align*}
for all $i\geq 0$.

By \cite[Theorem 7.2]{Kat70}, there exists the exact sequence
\begin{equation}
    0 \to B^{i}_U(\log\,B)\to Z^{i}_U(\log\,B)\xrightarrow{C^{i}_{U,B}} \Omega^{i}_U(\log\,B)\to 0,\label{log smooth 2}
\end{equation}
resulting from the logarithmic Cartier isomorphism.
Moreover, $B^{i}_U(\log\,B)$ and $Z^{i}_U(\log\,B) $ are locally free (cf.~\cite[Lemma 3.2]{Kaw4}).

\begin{defn}\label{def:reflexive Carter operators}
We define reflexive $\sO_X$-modules by 
\begin{align*}
    &B^{[i]}_X(\log B)\coloneqq j_{*}B^{i}_U(\log\,B)\,\,\text{and}\\
    &Z^{[i]}_X(\log B)\coloneqq j_{*}Z^{i}_U(\log B)
\end{align*}
for all $i\geq 0$.
The \textit{$i$-th reflexive Cartier operator}
\[
C^{[i]}_{X,B}\colon Z^{[i]}_X(\log\,B)\to \Omega^{[i]}_X(\log\,B)
\]
is defined as $j_{*}C^{i}_{U,B}$ for all $i\geq 0$.
\end{defn}

\subsection{Splitting sections and canonical liftings}
In this subsection, we recall splitting sections of Frobenius maps associated to $F$-liftings (\cite[Section 2]{BTLM}, \cite[Section 3.2]{AWZ}) and canonical liftings associated to splitting sections (\cite[Section 3.5]{Zda18}, \cite[Section A.1]{AZ21}).

\subsubsection{Splitting sections associated to $F$-liftings}\label{subsection:splitting section associated to F-lifting}

Let $X$ be a smooth $F$-liftable variety and $B$ a smooth prime divisor on $X$.
Fix a $W_2(k)$-lifting $(\tilde{X}, \tilde{B},\tilde{F}_X)$ of $(X,B,F)$.
Take $\tilde{f}\in \sO_{\tilde{X}}$. 
Since $\tilde{F}_X(\tilde{f})$ and $\tilde{f}^p$ coincide with each other on $X$, we have $\tilde{F}_X(\tilde{f})-\tilde{f}^p\in p\sO_{\tilde{X}}$.
We have an isomorphism $\sO_{X}\xrightarrow{p\times } p\sO_{\tilde{X}}$ resulting from the exact sequence
\[
0\to pW_2(k)\to W_2(k)\xrightarrow{p\times } pW_2(k)\to 0.
\]
We define $\delta_{\tilde{F}_X}(\tilde{f}) \in \sO_{X}$ as a corresponding element of $\tilde{F}_X(\tilde{f})-\tilde{f}^p\in p\sO_{\tilde{X}}$ via the above isomorphism, that is, 
    \[
    \delta_{\tilde{F}_X}\colon \sO_{\tilde{X}}\to \sO_{X};\,\, \tilde{f}\mapsto (\tilde{F}_X(\tilde{f})-\tilde{f}^p)/p.
    \] 

We have the pullback morphism
\begin{equation}\label{eq:pullback}
    \tilde{F}_X^{*}\colon\Omega_{\tilde{X}}^{i}(\log\,\tilde{B})\to (\tilde{F}_X)_{*}\Omega_{\tilde{X}}^{i}(\log\,\tilde{B})
\end{equation}
Then we can confirm that $\mathrm{Im}(\tilde{F}_X^{*})\subset p(\tilde{F}_X)_{*}\Omega_{\tilde{X}}^{i}(\log\,\tilde{B})$.
For example, let $\tilde{B}=\{\tilde{b}=0\}$ and we check $\tilde{F}_X^{*}(d\tilde{b}/\tilde{b})\in p(\tilde{F}_{\tilde{X}})_{*}\Omega_{\tilde{X}}^{i}(\log\,\tilde{B})$.
Since $\tilde{F}_X^{*}\tilde{B}=p\tilde{B}$, we have
$\tilde{F}_X^{*}\tilde{b}=\tilde{u}\tilde{b}^p$ for some $\tilde{u}\in \sO_{\tilde{X}}^{\times}$.
Then $\tilde{u}\tilde{b}^p=\tilde{F}_X^{*}\tilde{b}=\tilde{b}^p+p\tilde{\delta(\tilde{b})},$ where $\tilde{\delta(\tilde{b})}$ is a lift of $\delta(\tilde{b})$.
Thus $p\tilde{\delta(\tilde{b})}=(u-1)\tilde{b}^p \in \tilde{b}^p\sO_{\tilde{X}}$.
Since $\tilde{b}^p$ is a non-zero divisor and $\sO_{\tilde{X}}$ is flat over $W_{2}(k)$, we have $\Tor_1^{W_{2}(k)}(k, \sO_{\tilde{X}}/(\tilde{b}^p))=0$. 
Therefore, by the local criterion for flatness, $\sO_{\tilde{X}} / (\tilde{b}^p)$ is flat over $W_{2}(k)$, and we have 
\[
\tilde{b}^p\sO_{\tilde{X}} \cap p \sO_{\tilde{X}} = p \tilde{b}^p \sO_{\tilde{X}}.
\]
Thus, we have $p\tilde{\delta(\tilde{b})} \in p \tilde{b}^p \sO_{\tilde{X}}$, and $\delta(\tilde{b}) =  b^p a$ for some $a \in \sO_{X}$, where $b = \tilde{b} \mod p$.
Therefore, $d\delta(\tilde{b})=b^pda$, and
we have $d\tilde{\delta(\tilde{b})}-\tilde{b}^pd\tilde{a}\in p\Omega^{i}_{\tilde{X}}$.
Now, we obtain
\[
d\tilde{F}_X^{*}\tilde{b}/\tilde{F}_X^{*}\tilde{b}
=p(\tilde{b}^{p-1}d\tilde{b}+d\tilde{\delta(\tilde{b})})/\tilde{u}\tilde{b}^p
=p\tilde{u}^{-1}(d\tilde{b}/\tilde{b}+d\tilde{a})\in p(\tilde{F}_X)_{*}\Omega_{\tilde{X}}^{i}(\log\,\tilde{B}),
\]
as desired.

Using $(F_X)_{*}\Omega_{X}^{i}(\log\,B)\cong p(\tilde{F}_X)_{*}\Omega_{\tilde{X}}^{i}(\log\,\tilde{B})$, the following map is induced from \eqref{eq:pullback}:
\[
\xi^{i}_{X,B}\colon \Omega_X^{i}(\log\,B)\to (F_X)_{*}\Omega_X^{i}(\log\,B).
\]
Moreover, $\xi^{i}_{X,B}$ is a split injection by \cite[Proposition 3.2.1 and Variant 3.2.2]{AWZ}.
Similarly, the pullback morphisms
\[
(\tilde{F}_X)^{*}\colon\Omega_{\tilde{X}}^{i}\to (\tilde{F}_X)_{*}\Omega_{\tilde{X}}^{i}
\]
\[
(\tilde{F}_B)^{*}\colon\Omega_{\tilde{B}}^{i}\to (\tilde{F}_B)_{*}\Omega_{\tilde{B}}^{i}
\]
induce split injections
\[
\xi^{i}_{X}\colon \Omega_X^{i}\to (F_X)_{*}\Omega_X^{i}
\]
\[
\xi^{i}_{B}\colon \Omega_B^{i}\to (F_B)_{*}\Omega_B^{i}.
\]
Then we have the following commutative diagram:
\begin{equation*}
\xymatrix{ 0\ar[r]&\Omega^{i}_X\ar[r]\ar[d]^-{\xi^{i}_{X}} & \Omega^{i}_X(\log\,B)\ar[r]^-{\mathrm{res}}\ar[d]^-{\xi^{i}_{X,B}}& \Omega^{i-1}_B \ar[r]\ar[d]^-{\xi^{i-1}_{B}}& 0\\
           0\ar[r]&(F_X)_{*}\Omega^{i}_X\ar[r] & (F_X)_{*}\Omega^{i}_X(\log\,B)\ar[r]^-{\mathrm{res}}& (F_B)_{*}\Omega^{i-1}_B \ar[r]& 0
           .}
\end{equation*}
Taking $i=\dim\,X$ and applying $R\sHom_{\sO_X}(-,\omega_X)$, we have the following commutative diagram:
\begin{equation*}
\xymatrix{ 0\ar[r]&\sO_X(-B)\ar[r] & \sO_X\ar[r]& \sO_B \ar[r]& 0\\
           0\ar[r]&(F_X)_{*}\sO_X(-B)\ar[r]\ar[u]^-{(\xi^{\dim\,X}_{X,B})^{*}} & (F_X)_{*}\sO_X\ar[r]\ar[u]^-{(\xi^{\dim\,X}_{X})^{*}}& (F_B)_{*}\sO_B \ar[r]\ar[u]^-{(\xi^{\dim\,B}_{B})^{*}}& 0
           .}
\end{equation*}
We set $\sigma(\tilde{F}_X)\coloneqq (\xi^{\dim\,X}_{X})^{*}$ and $\sigma(\tilde{F}_B)\coloneqq (\xi^{\dim\,B}_{X})^{*}$.
Since they provide splitting sections of Frobenius maps (see \cite[Proposition 3.2.1 and Variant 3.2.2]{AWZ}), we call them \textit{splitting sections associated to $\tilde{F}_X$ and $\tilde{F}_B$}, respectively.
Therefore, we obtain a Frobenius splitting of $X$ compatible with $B$:
\begin{equation*}
\xymatrix{ \sO_X\ar[r]& \sO_B \\
           (F_X)_{*}\sO_X\ar[r]\ar[u]^-{\sigma(\tilde{F}_X)}& (F_B)_{*}\sO_B\ar[u]^-{\sigma(\tilde{F}_B)}
           .}
\end{equation*}
\subsubsection{Canonical liftings associated to splitting sections}\label{section:canonical lift}

Let $X$ be a globally $F$-split variety. We fix a splitting $\sigma\colon F_{*}\sO_X\to \sO_X$ of the Frobenius map.
We define a $W_2\sO_X$-module $\sO_{\tilde{X}(\sigma)}$ by the following pushout:
\[
\xymatrix{
F_{*}\sO_X \ar[r]^-{V} \ar[d]_{\sigma} & W_2\sO_X \ar[d]^{\pi_X}  \\
\sO_X \ar[r] &\sO_{\tilde{X}(\sigma)}\pushoutcorner
}
\]
where $V\colon F_{*}\sO_X\to W_2\sO_X$ denotes the Verschiebung.
The sheaf $\sO_{\tilde{X}(\sigma)}$ is a sheaf of rings since $\sO_{\tilde{X}(\sigma)}=W_2\sO_X/V(\mathrm{Ker}(\sigma))$ by the definition of a pushout and $V(\mathrm{Ker}(\sigma))\subset W_2\sO_X$ is an ideal (cf.~\cite[Lemma A.1.1]{AZ21}). 
We set $\tilde{X}({\sigma})\coloneqq (|X|, \sO_{\tilde{X}(\sigma)})$.
Then we can confirm that $\tilde{X}({\sigma})$ is a $W_2(k)$-lifting of $X$ (see \cite[Corollary A.1.2]{AZ21}).
We call $\tilde{X}(\sigma)$ a \textit{canonical lift associated to a splitting section $\sigma$}.

Suppose that $X$ is smooth and $F$-liftable. Fix a $W_2(k)$-lifting $(\tilde{X}, \tilde{F}_X)$ of $(X,F_X)$.
As in the previous section, we can take a splitting section $\sigma(\tilde{F}_X)$ associated to $\tilde{F}_X$.
In this case, we call $\tilde{X}(\sigma(\tilde{F}_X))$ a \textit{canonical lift associated to $\tilde{F}_X$}, and denote it by $\tilde{X}(\tilde{F}_X)$.

\subsection{Characterization of $F$-liftability via reflexive Cartier operators}

In this subsection, we prove Theorem \ref{thm:characterize F-lift via Cartier operators}, which asserts that $F$-liftability can be characterized by a splitting of the reflexive Cartier operator. For the proof, we use a similar technique to \cite[Theorem 1.13]{Langer24}.

\begin{thm}\label{thm:characterize F-lift via Cartier operators}
   Let $(X,B)$ be a pair of a normal variety $X$ and a reduced divisor $B$.
   Then the following are equivalent.
   \begin{enumerate}
       \item[\textup{(1)}] $(X,B)$ is $F$-liftable.
       \item[\textup{(2)}] $C_{X,B}^{[i]}\colon Z^{[i]}_{X}(\log\,B)\to \Omega^{[i]}_{X}(\log\,B)$ are split surjections for all $i\geq 0$.
       \item[\textup{(3)}] $C_{X,B}^{[1]}\colon Z^{[1]}_{X}(\log\,B)\to \Omega^{[1]}_{X}(\log\,B)$ is a split surjection.
   \end{enumerate}
\end{thm}

\begin{lem}\label{lem:induced}
  Let $U$ be a smooth variety and $B$ a smooth prime divisor on $U$.
Suppose that $(U,B_U)$ is $F$-liftable, and
fix a $W_2(k)$-lifting $(\tilde{U}, \tilde{B_U}, \tilde{F}_U)$ of $(U,B_U,F_U)$.
Then 
$\tilde{F}_U|_{\tilde{B_U}}\colon \tilde{B_U}\to \tilde{B_U}$ is induced.
\end{lem}
\begin{proof}
    Since $\tilde{F}_U^{*}\sO_{\tilde{U}}(-\tilde{B_U})=\sO_{\tilde{U}}(-p\tilde{B_U})\subset \sO_{\tilde{U}}(-\tilde{B_U})$, the desired map is induced.
\end{proof}

\begin{lem}\label{lem:compatibility}
    Let $(U,B_U)$ be an $F$-liftable log smooth pair.
    We fix a $W_2(k)$-lifting $(\tilde{U}, \tilde{B_U},\tilde{F}_U)$of $(U,B_U,F_U)$.
    Let $\tilde{F}_{B_U}\coloneqq \tilde{F}_U|_{\tilde{B_U}}$ be the induced lifting of $F_{B_U}$ (Lemma \ref{lem:induced}).
    Let $\tilde{U}(\tilde{F}_U)$ and $\tilde{B}(\tilde{F}_{B_U})$ be the canonical liftings of $X$ and $B$ induced by $\tilde{F}_U$ and $\tilde{F}_{B_U}$.
    Then there exists a lift  $\tilde{F}'_U\colon \tilde{U}(\tilde{F}_U)\to \tilde{U}(\tilde{F}_U)$ of $F_U$
   such that $(\tilde{F}'_U)^{*}\tilde{B_U}(\tilde{F}_{B_U})=p\tilde{B_U}(\tilde{F}_{B_U})$.
\end{lem}
\begin{proof}
We first assume that $B_U$ is a prime divisor.
There exists a canonical isomorphism $\tilde{U}(\tilde{F}_U)\to \tilde{U}$ by \cite[Theorem 2.7]{AWZ2}.
We will show that this isomorphism induces an isomorphism $\tilde{B_U}(\tilde{F}_{B_U})\to \tilde{B_U}$.
Then we can take $\tilde{F}'_U$ as a map corresponds to $\tilde{F}_U$.

We define ring homomorphisms $v_{\tilde{F}_{U}}$ and $v_{\tilde{F}_{B_U}}$ by
    \[
    v_{\tilde{F}_{U}}\colon  \sO_{\tilde{U}}\to W_2\sO_{U}; \tilde{f}\mapsto (\tilde{f}\,\,\text{mod}\,p, \delta_{\tilde{F}_{U}}(\tilde{f}))
    \] and 
    \[v_{\tilde{F}_{B_U}}\colon\sO_{\tilde{B_U}}\to W_2\sO_{B_U}; \tilde{b}\mapsto(\tilde{b}\,\,\text{mod}\,p, \delta_{\tilde{F}_{U_B}}(\tilde{b})),
    \] 
    where we refer to Section \ref{subsection:splitting section associated to F-lifting} for the definition of $\delta$.
    By the construction of canonical liftings (Section \ref{section:canonical lift}), we have the natural quotient maps
    \[
    \pi_{\tilde{F}_U}\colon W_2\sO_U\to \sO_{\tilde{U}(\tilde{F}_U)}
    \]
    and 
    \[
    \pi_{\tilde{F}_{B_U}}\colon W_2\sO_{B_U}\to \sO_{\tilde{B_U}(\tilde{F}_{B_U})}.
    \]
    Then we have the following commutative diagram:
    \begin{equation*}
\xymatrix{ \sO_{\tilde{U}}\ar[r]^-{v_{\tilde{F}_{U}}}\ar[d] & W_2\sO_U\ar[r]^-{\pi_{\tilde{F}_U}}\ar[d] & \sO_{\tilde{U}(\tilde{F}_U)}\ar[d] \\
            \sO_{\tilde{B_U}}\ar[r]^-{v_{\tilde{F}_{B_U}}} \ar[r] &W_2\sO_{B_U}\ar[r]^-{\pi_{\tilde{F}_{B_U}}}   & \sO_{\tilde{B_U}(\tilde{F}_{B_U})}.
            }
\end{equation*}
    As proven in \cite[Theorem 2.7]{AWZ2}, the composition of top morphisms is an isomorphism, and the same proof shows that the composition of bottom maps is also isomorphism.
    Therefore, the assertion holds when $B_U$ is prime.
    In a general case, since $v_{\tilde{F}_{U}}$ does not depend on $B_U$, applying the argument above to each component of $B_U$, we can conclude.
\end{proof}

\begin{proof}[Proof of Theorem \ref{thm:characterize F-lift via Cartier operators}]
    (1)$\Rightarrow $(3) follows from the proof of \cite[Lemma 3.8]{Kaw4} and (3)$\Rightarrow$(2) is obvious.
    We prove (2)$\Rightarrow $(1).
    Let $U$ be the log smooth locus of $(X,B)$ and $B_U\coloneqq B|_{U}$.
    Since the exact sequence  
    \[
    0\to B_U^{1}(\log\,B_U) \to Z_U^{1}(\log\,B_U) \to \Omega_U^{1}(\log\,B_U)\to 0
    \]
   splits, the pair $(U, B_U)$ is $F$-liftable by \cite[Proposition 3.2.1 and Variant 3.2.2]{AWZ}. 
   We first assume that $B$ is prime.
   We fix $W_2(k)$-liftings $(\tilde{U}, \tilde{B_U}, \tilde{F}_U)$ of $(U,B_U,F_U)$.
   Let $\tilde{F}_{B_{U}}\coloneqq \tilde{F}_U|_{\tilde{U_B}}$ be the induced map on $\tilde{B_U}$, which is a lift of $F_{B_{U}}$ (Lemma \ref{lem:induced}).
   Let $\sigma(\tilde{F}_U)\colon F_{*}\sO_U\to \sO_U$ and $\sigma(\tilde{F}_{B_U})\colon F_{*}\sO_{B_U}\to \sO_{B_U}$ be the splitting sections associated to $\tilde{F}_U$ and $\tilde{F}_{B_U}$, respectively.
   Then as we have confirmed in Section \ref{subsection:splitting section associated to F-lifting}, we have the following commutative diagram:
    \begin{equation*}
\xymatrix{ F_{*}\sO_{U}\ar[r]^-{\sigma(\tilde{F}_U)}\ar[d] & \sO_U\ar[d] \\
            F_{*}\sO_{B_U}\ar[r]^-{\sigma(\tilde{F}_B)} \ar[r] & \sO_{B_U}.}
\end{equation*}
   Since $X$ is normal, taking the pushforward by the inclusion $j\colon U\hookrightarrow X$, we obtain a splitting
   \[
   \sigma_X\coloneqq j_{*}\sigma(\tilde{F}_U)\colon  F_{*}\sO_X\to \sO_X
   \]
   of $F_X$.
   By \cite[Lemma 1.1.7 (ii)]{fbook}, the map $\sigma(\tilde{F}_{B_U})\colon F_{*}\sO_{B_U}\to \sO_{B_U}$ extends to $\sigma_B\colon F_{*}\sO_{B}\to \sO_{B}$ so that the following diagram commutes:
    \begin{equation*}
\xymatrix{ F_{*}\sO_{X}\ar[r]^-{\sigma_X}\ar[d] & \sO_X\ar[d] \\
            F_{*}\sO_{B}\ar[r]^-{\sigma_B} \ar[r] & \sO_{B}.}
\end{equation*}
   Let $\tilde{X}(\sigma_X)$ and $\tilde{B}({\sigma_B})$ be the canonical liftings of $X$ and $B$ associated to the splitting sections $\sigma_X$ and $\sigma_B$. 
   Recalling the construction of canonical liftings (Section \ref{section:canonical lift}),
   we have the following commutative diagram:
   \begin{equation*}
\xymatrix{ \tilde{B_U}(\tilde{F}_{B_U})\coloneqq \tilde{B_U}(\sigma(\tilde{F}_{B_U}))\ar@{^{(}->}[r]\ar@{^{(}->}[d] & \tilde{B}(\sigma_B)\ar@{^{(}->}[d] \\
            \tilde{U}(\tilde{F}_U)\coloneqq \tilde{U}(\sigma(\tilde{F}_U)) \ar@{^{(}->}[r] & \tilde{X}(\sigma_X).}
\end{equation*}
   By Lemma \ref{lem:compatibility}, we have a lift  $\tilde{F}'_U\colon \tilde{U}(\tilde{F}_U)\to \tilde{U}(\tilde{F}_U)$ of $F_U$
   such that $(\tilde{F}'_U)^{*}\tilde{B_U}(\tilde{F}_{B_U})=p\tilde{B_U}(\tilde{F}_{B_U})$.
   Moreover, $\tilde{F}'_U$ extends to a lift $\tilde{F}_X\colon \tilde{X}(\sigma_X)\to \tilde{X}(\sigma_X)$ of $F_X$ by \cite[Theorem 3.3.6 (a)-(iii)]{AWZ} since $\tilde{U}(\tilde{F}_U)\hookrightarrow \tilde{X}(\sigma_X)$ is a lift of the inclusion $U\hookrightarrow X$. 
   Then we have
   \[
   \tilde{F}_X^{*}(\tilde{B}(\sigma_B)|_{\tilde{U}(\tilde{F}_U)})=(\tilde{F}'_U)^{*}\tilde{B}_U(\tilde{F}_{B_U})=p\tilde{B_U}(\tilde{F}_{B_U})=p(\tilde{B}(\sigma_B)|_{\tilde{U}(\tilde{F}_U)}).
   \]
   Thus, we conclude when $B$ is prime.
   In a general case, since the construction of $\tilde{X}$ and $\tilde{F}_X$ does not depend on $B$, applying the argument above to each component of $B$, we can conclude.
\end{proof}

\begin{cor}\label{cor:local=global for affine}
    Let $(X,B)$ be a pair of a normal variety $X$ and a reduced divisor $B$ such that $(X,B)$ is locally $F$-liftable.
    If $X$ is affine, then $(X,B)$ is $F$-liftable.
\end{cor}
\begin{proof}
    Since $(X,B)$ is locally $F$-liftable, the map $C^{[1]}_{X,B}\otimes \sO_{X,P}$ is a split surjection at every closed point $P$.
    Thus, $C^{[1]}_{X,B}$ splits globally if $X$ is affine.
\end{proof}

\begin{cor}\label{cor:F-lift to F-pure}
    Let $(X,B)$ be a pair of a normal variety $X$ and a reduced divisor $B$.
    If $(X,B)$ is locally $F$-liftable, then it is $F$-pure.
\end{cor}
\begin{proof}
    We fix a closed point $P\in X$.
    By Theorem \ref{thm:characterize F-lift via Cartier operators} (1)$\Rightarrow$(2), the map \[
    C_{X,B}^{[d]}\colon F_{*}\omega_X(B) \to \omega_X(B)
    \]
    is a split surjection at $P\in X$.
    Taking $\sHom_{\sO_X}(-,\omega_X(B))$, we have a map
    \[
    \sO_X  \to F_{*}\sO_X((p-1)B),
    \]
    which is a split injection at $P$. Thus, $(X,B)$ is $F$-pure at $P$.
\end{proof}

\begin{rem}
    We say a singularity $(P\in X,B)$ is $F$-liftable if some representative $(P\in X,B)$ is $F$-liftable at $P$.
    Since $C_{X,B}^{[1]}\otimes_{\sO_X} \sO_{X,P}$ splits if and only if so does $C_{X,B}^{[1]}\otimes_{\sO_X} \sO_{X,P}^{\wedge}$, the definition is equivalent to saying that every representative $(P\in X,B)$ is $F$-liftable at $P$.
\end{rem}

\subsection{Some descent properties}\label{sec:descent}

As one of the advantages of the characterization of $F$-liftability via reflexive Cartier operators, we obtain some descent properties.

\begin{prop}[Birational descent]\label{lem:birational F-lift descend}
    Let $f\colon Y\to X$ be a proper birational morphism of normal varieties.
    Let $B_Y$ be a reduced divisor on $Y$ and $B\coloneqq f(B_Y)$.
    If $(Y,B_Y)$ is $F$-liftable, then so is $(X,B)$.
\end{prop}
\begin{proof}
    Since $(Y,B_Y)$ is $F$-liftable, the reflexive Cartier operator
    \[
    C_{Y,B_Y}^{[1]}\colon Z_Y^{[1]}(\log\,B_Y) \to \Omega_Y^{[1]}(\log\,B_Y)
    \]
    splits by Theorem \ref{thm:characterize F-lift via Cartier operators} (1)$\Rightarrow$ (3).
    Taking the pushforward by $f$ and taking double dual, we obtain a splitting of the reflexive Cartier operator
    \[
    C_{X,B}^{[1]}\colon Z_X^{[1]}(\log\,B) \to \Omega_X^{[1]}(\log\,B).
     \]
    Thus, $(X,B)$ is $F$-liftable by Theorem \ref{thm:characterize F-lift via Cartier operators} (3)$\Rightarrow$ (1).
\end{proof}

\begin{lem}\label{lem:split for BZOmega}
    Let $g\colon X'\to X$ be a finite morphism of normal varieties of degree prime to $p$.
    Let $B$ a reduced divisor on $X$ and $B'$ be the sum of components of $g^{-1}(B)$ along which the ramification indices are prime to $p$.
    Then the natural pullback maps
    \begin{equation}
        g^{*}\colon\Omega_X^{[i]}(\log\,B) \to g_{*}\Omega_{X'}^{[i]}(\log\,B')\label{split:Omega}
    \end{equation}
    \begin{equation}
        g^{*}\colon Z_X^{[i]}(\log\,B) \to g_{*}Z_{X'}^{[i]}(\log\,B')\label{split:Z}
    \end{equation}
    \begin{equation}
        g^{*}\colon B_X^{[i]}(\log\,B) \to g_{*}B_{X'}^{[i]}(\log\,B')\label{split:B}
    \end{equation}
    split for all $i\geq 0$.
\end{lem}
\begin{proof}
    A splitting of \eqref{split:Omega} is proved in \cite[Lemma 2.1]{Totaro(logBott)}.
    We first prove a splitting of \eqref{split:Z}.
    The map $g^{*}\colon Z_X^{[i]}(\log\,B) \to g_{*}Z_{X'}^{[i]}(\log\,B')$ is induced as follows:
    \begin{equation*}
\xymatrix{ 0\ar[r] &Z_X^{[i]}(\log\,B)\ar[r]^-{F_{*}d}\ar@{.>}[d]^-{g^{*}}& F_{*}\Omega_X^{[i]}(\log\,B)\ar[r]^-{F_{*}d}\ar[d]^-{F_{*}g^{*}}& F_{*}\Omega^{[i+1]}_X(\log\,B)\ar[d]^-{F_{*}g^{*}} \\
               0\ar[r] & g_{*}Z^{[i]}_{X'}(\log\,B')\ar[r]^-{F_{*}d} &g_{*}F_{*}\Omega^{[i]}_{X'}(\log\,B')\ar[r]^-{F_{*}d}   & g_{*}F_{*}\Omega^{[i+1]}_{X'}(\log\,B').}
\end{equation*}
Then it is easy to confirm that a splitting of the middle and the right vertical maps induce the splitting of $\eqref{split:Z}$.

    Similarly, the map $g^{*}\colon B_X^{[i]}(\log\,B) \to g_{*}B_{X'}^{[i]}(\log\,B')$ is induced as follows:
    \begin{equation*}
\xymatrix{ Z_X^{[i-1]}(\log\,B)\ar[r]^-{F_{*}d}\ar[d]^-{g^{*}}& F_{*}\Omega_X^{[i-1]}(\log\,B)\ar[r]^-{F_{*}d}\ar[d]^-{F_{*}g^{*}}& B_X^{[i]}(\log\,B)\ar@{.>}[d]^-{g^{*}}\ar[r]& 0 \\
                g_{*}Z^{[i-1]}_{X'}(\log\,B')\ar[r]^-{F_{*}d} &g_{*}F_{*}\Omega^{[i-1]}_{X'}(\log\,B')\ar[r]^-{F_{*}d}   & g_{*}B^{[i]}_{X'}(\log\,B')\ar[r] &0.}
\end{equation*}
Then we can confirm that a splitting of the left and the middle vertical maps induce a splitting of \eqref{split:B}.
\end{proof}

\begin{thm}[Finite descent]\label{thm:finite descent}
    Let $g\colon X'\to X$ be a finite morphism of normal varieties of degree prime to $p$.
    Let $B$ be a reduced divisor on $X$ and $B'$ be the sum of components of $g^{-1}(B)$ along which the ramification index prime to $p$.
    If $(X',B')$ is $F$-liftable, then so is $(X,B)$.
\end{thm}
\begin{rem}
    If $g$ is a quotient by a finite group scheme of degree prime to $p$, then the assumption about the ramification index is satisfied.
\end{rem}
\begin{proof}
    Suppose that $(X',B')$ is $F$-liftable.
    Then we have a map 
    \[
    \phi\colon \Omega_{X'}^{[i]}(\log\,B') \to Z_{X'}^{[i]}(\log\,B'),
    \]
    which gives a split of $C_{X',B'}^{[i]}$.
    By Lemma \ref{lem:split for BZOmega},
    we have a map
    \[
    \psi\colon g_{*}Z_{X'}^{[i]}(\log\,B') \to Z_{X}^{[i]}(\log\,B),
    \]
    which gives a split of the pullback map 
    \[
    g^{*}\colon Z_{X}^{[i]}(\log\,B)\to g_{*}Z_{X'}^{[i]}(\log\,B').
    \]
    Then the composition of 
    \[
    \Omega_X^{[i]}(\log\,B)\xrightarrow{g^{*}} g_{*}\Omega_{X'}^{[i]}(\log\,B')\xrightarrow{g_{*}\phi} g_{*}Z_{X'}^{[i]}(\log\,B')\xrightarrow{\psi}Z_{X}^{[i]}(\log\,B) ,
    \]
    gives a splitting of 
    \[
    C^{[i]}_{X,B}\colon Z_{X}^{[i]}(\log\,B)\to \Omega_X^{[i]}(\log\,B),
    \]
    as desired.
\end{proof}

Let $g\colon X'\to X$ be a finite morphism of normal varieties of degree prime to $p$.
Achinger-Witaszek-Zdanowicz \cite[Theorem 3.3.6 (a)-(ii)]{AWZ} proved the descent of $F$-liftability when $g$ lifts to $W_2(k)$.
Taking $B=0$ in Theorem \ref{thm:finite descent}, we can prove the descent of $F$-liftability without assuming the liftability of $g$. 

\begin{cor}\label{cor:finite descent}
    Let $g\colon X'\to X$ be a finite morphism of normal varieties of degree prime to $p$.
    If $X'$ is $F$-liftable, then so is $X$.
\end{cor}
\begin{proof}
    Taking $B=0$ in Theorem \ref{thm:finite descent}, we obtain the assertion.
\end{proof}


\section{Local \texorpdfstring{$F$}--liftability in dimension two}
In this section, we prove Theorem \ref{Introthm:F-purity and F-liftability}.

\subsection{\texorpdfstring{$F$}--liftability of RDPs}\label{subsec:F-lift for RDPs}

The main tool is Zdanowicz's criterion \cite[Corollary 4.12]{Zda18} for $F$-liftability of hypersurfaces and his program \cite[Section 8.2]{Zda18} of Macaulay 2. On the other hand, when we deal with RDPs of type $D_{2n}^{n-1}$ and $D_{2n+1}^{n-1}$, it is not easy to verify their $F$-liftability in this way; for a fixed $n$, we can check the $F$-liftability by the program, but $n$ can be arbitrarily large. 
To deal with this problem, we prove Lemma \ref{lem:F-lift criterion}, which enables us find elements which are necessary in his criterion for all $n$ simultaneously.

In this subsection, we use the following notation.
\begin{conv}
\label{conv:delta}\,
\begin{itemize}
    \item $R\coloneqq k[x_1,\ldots, x_n]$.
    \item $I_f\coloneqq (\frac{\partial f}{\partial x_1}, \ldots, \frac{\partial f}{\partial x_n}) \subset R$  for $f\in R$.
    \item $\Delta_1(f)=\sum_{\substack{0 \leq \alpha_1, \ldots,\alpha_m \leq p-1 \\ \alpha_1+\cdots+\alpha_m=p}} \frac{1}{p} \binom{p}{\alpha_1, \ldots ,\alpha_m}(M_1)^{\alpha_1} \cdots (M_m)^{\alpha_m}$, where $f \in R$ and $f=\sum^m_{i=1} M_i$ is the monomial decomposition.
\end{itemize}
 \end{conv}

\begin{thm}[\textup{\cite[Corollary 4.12]{Zda18}}]\label{thm:Zdanowicz Critierion}
    A spectrum of $R/f$ is $F$-liftable if and only if there exists $g\in R$ such that 
    \[
    \Delta_1(f)+g^p\in (f, I_f^{[p]}).
    \]
\end{thm}

\begin{defn}
    We say that $\sigma\in \Hom_R(F_{*}R,R)$ is a splitting section compatible with an ideal $(f)$ if 
    $\sigma(1)=1$ and $\sigma(F_{*}(f)) \subset (f)$.
\end{defn}

The following lemma was taught to authors by Shou Yoshikawa.
This lemma gives a canonical way to find $g$ in Theorem \ref{thm:Zdanowicz Critierion}.
We note that if $R/f$ is normal and $F$-liftable, then $R/f$ is $F$-pure (Corollary \ref{cor:F-lift to F-pure}). 
Then we can take a splitting section compatible with an ideal $(f)$ since $F_{*}R$ is a free $R$-module.

\begin{lem}\label{lem:F-lift criterion}
    Let $f\in R$ and $\sigma\in \Hom_R(F_{*}R,R)$ is a splitting section compatible with an ideal
    $(f)$.
    Then $\Spec\,R/f$ is $F$-liftable if and only if 
    \[
    \Delta_1(f)+(-\sigma(F_{*}\Delta_1(f)))^p\in (f, I_f^{[p]}).
    \]
\end{lem}
\begin{proof}
Suppose that there exists $g\in R$ such that
$\Delta_1(f)+g^p\in (f,I^{[p]}_{f})$.
Then $F_{*}\Delta_1(f)+F_{*}g^p\in F_{*}(f,I^{[p]}_{f})$.
Sending by $\sigma$, 
we have $\sigma(F_{*}\Delta_1(f))+g\in (f,I_{f})$, and in particular, $\sigma(F_{*}\Delta_1(f))^{p}+g^{p}\in (f,I^{[p]}_{f})$.
Therefore, we obtain 
\[
\Delta_1(f)+(-\sigma(F_{*}\Delta_1(f)))^p=(\Delta_1(f)+g^p)-(\sigma(F_{*}\Delta_1(f))^{p}+g^{p})\in (f,I^{[p]}_{f}),
\]
as desired.
\end{proof}

\begin{rem}\label{rem:F-lift ci}
More generally, by the same method, we can show the following: Let $R/(f_1, \ldots, f_m)$ be a complete intersection algebra, and $\sigma\in \Hom_R(F_{*}R,R)$ a splitting section compatible with an ideal $(f_1, \ldots, f_m)$.
Then $R/(f_1, \ldots, f_m)$ is $F$-liftable if and only if there exists a sequence of elements $h_k \in R$ such that 
\[
\Delta_1 (f_i) + (-\sigma(F_{*}(\Delta_1(f_i))))^p = \sum_{1\leq k \leq n} \left(\frac{\partial f_i}{ \partial x_k}\right)^p h_k  
\mod (f_1, \ldots, f_m)
\]
for any $1\leq i \leq m.$
Indeed, by \cite[Theorem 4.10]{Zda18}, the $F$-liftability is equivalent to the exitence of a sequence of elements $h_k, g_i \in R$ such that
\[
\Delta_1 (f_i) + g_i^p = \sum_{1\leq k \leq n} \left(\frac{\partial f_i}{ \partial x_k}\right)^p h_k  
\mod (f_1, \ldots, f_m)
\]
for any $1\leq i \leq m.$
By applying $\sigma \circ F_*$ and taking the $p$-th power, we have
\[
\sigma(F_* \Delta_1 (f_i))^p + g_i^p = \sum_{1 \leq k \leq n} \left(\frac{\partial f_i}{ \partial x_k}\right)^p \sigma (F_* h_k)^p \mod (f_1, \ldots, f_m).
\]
Therefore, by replacing $h_i$, we obtain the desired equality.
\end{rem}

\begin{rem}
\label{rem:fliftpccriteria}
In Lemma \ref{lem:F-lift criterion}, we use an actual splitting $\sigma$, which is not easy to find in general.
However, if $R/f$ has an isolated singularity at $0 = (x_1, \ldots, x_n)$, we can loosen the condition of $\sigma$ as the following:
For any $f\in R$ and $\sigma\in \Hom_R(F_{*}R,R)$ satisfying $\sigma(1) \in R_{0}^{\times}$ and $\sigma (F_{\ast}(f)) \subset (f)$,
the spectrum $\Spec R/f$ is $F$-liftable if and only if
\begin{equation}
\label{eq:pccriteria}
(\sigma(F_*1))^p\Delta_1(f)+(-\sigma(F_{*}\Delta_1(f)))^p\in (f, I_f^{[p]}).
\end{equation}
Indeed, if this condition holds true, then by the same method in \cite[Corollary 4.12]{Zda18}, we can show that $\Spec (R/f)_{0}$ is $F$-liftable. Since $R$ has an isolated singularity, we can show that $\Spec R/f$ is $F$-liftable.
Conversely, if $\Spec R/f$ is $F$-liftable, the same method as in Lemma \ref{lem:F-lift criterion} gives the inclusion (\ref{eq:pccriteria}).

The criterion (\ref{eq:pccriteria}) is suitable for computer algebra system.
Indeed, such a $\sigma$ can be captured easily by using Fedder's criteria, and we can calculate $F$-liftability faster than the procedure in \cite{Zda17}.
\end{rem}

\begin{prop}\label{prop:canonical}
Let $(P\in X)$ be an $F$-pure canonical singularity.
If $p\neq 5$, then $(P\in X)$ is $F$-liftable.
When $p=5$, the singularity $(P\in X)$ is $F$-liftable if and only if it is not an RDP of type $E_8^1$.
\end{prop}
\begin{proof}
By Artin’s approximation theorem, we may assume that $X=\Spec\,R/f$, where $f$ is one of equations of RDPs given in \cite{Artin(RDP)}. 

Let $P\in X$ be a singular point. It suffices to show the following: 
\begin{enumerate}
\item[\textup{(1)}]
If $p=2$ and $(P\in X)$ is of type $D_{2n}^{n-1}(n\geq 2)$, $D_{2n+1}^{n-1}(n\geq 2)$, $E_{6}^{1}$, $E_{7}^{3}$, or $E_{8}^4$, then $(P\in X)$ is $F$-liftable.
\item[\textup{(2)}]
If $p=3$ and $(P\in X)$ is of type $E_{6}^{1},$ $E_{7}^{1}$, or $E_{8}^{3}$, then $(P\in X)$ is $F$-liftable.
\item[\textup{(3)}]
If $p=5$ and $(P\in X)$ is of type $E_{8}^{1}$, then $(P\in X)$ is not $F$-liftable.
\end{enumerate}

Except for the case where $p=2$ and $P$ is of type $D_{2n}^{n-1}$ or $D_{2n+1}^{n-1}$, the assertion follows from Theorem \ref{thm:Zdanowicz Critierion} and Zdanowicz's program of Macaulay 2 in \cite[Section 8.2]{Zda17}.

Suppose that $p=2$ and $P$ is of type $D_{2n}^{n-1}$.
We put
\[
f\coloneqq z^2+x^2y+xy^n+xyz.
\]
We put $(f_x,f_y,f_z)\coloneqq(\frac{\partial f}{\partial x}, \frac{\partial f}{\partial y}, \frac{\partial f}{\partial z})$.
By Lemma \ref{lem:F-lift criterion}, it suffices to prove that 
\[
 \Delta_1(f)+(-\sigma(F_{*}\Delta_1(f)))^p \in (f, f_{x}^2, f_{y}^2, f_{z}^2)
 \]
for some splitting $\sigma \in \Hom_{R}(F_{*}R,R)$.
We can take $\sigma$ as
\[
\sigma (F_* r) = u(F_{\ast}(rf)),
\]
where $u\in \Hom_{R} (F_{*}R,R)$ is defined by
\[
u(F_{*}x^iy^jz^k)
\coloneqq
\begin{cases}
x^{\frac{i-1}{2}}y^{\frac{j-1}{2}}z^{\frac{k-1}{2}} & \textup{ if } i,j, \textup{ and }k \textup{ are odd,} \\
0 & \textup{ otherwise.}
\end{cases}
\]
Then, by direct computation, we can check that
    \begin{align*}
  &  \Delta_1(f)+(-\sigma(F_{*}\Delta_1(f)))^p\\
    =& z^2f  + (x^2+xy^{n-1}+ xz+y^{2n-2}+y^{n-1}z) f_{z}^2,
    \end{align*}
and we have the desired result.

Next, suppose that the singularity is of type $D_{2n+1}^{n-1}$.
We put 
\[
f\coloneqq z^2+x^2y+y^nz+xyz,
\]
and $\sigma \coloneqq u(F_{\ast} (-\times f))$ as before.
Then we can check that
\begin{align*}
   & \Delta_1(f)+(-\sigma(F_{*}\Delta_1(f)))^p\\
    =& (z^2+y^{2n-1})f  + (x^2+xy^{n-1}+ y^{2n-3}) f_{x}^2 +(x^2+z^2+xz+y^{n-1}z)f_{z}^2,
\end{align*}
and it finishes the proof.
\end{proof}

\subsection{\texorpdfstring{$F$}--liftability of lc surface singularities}

Next, we investigate $F$-liftability for lc surface singularities that are not RDPs.
In this section, we freely use the classification of lc surface singularities \cite[Section 3.3]{Kol13}.
The list in \cite[Section 7.B]{Gra} is also useful.

\begin{defn}
Let $(X,B)$ be a pair of a normal variety and a reduced divisor.
We say a singularity $(P\in X,B)$ is \textit{toric} if $(X,B)$ has a common \'etale neighborhood with some toric pair.
\end{defn}

\begin{lem}\label{lem:cyc quotient are toric}
    A cyclic quotient singularity $(P\in X,B)$ \cite[3.40.1]{Kol13} is a toric singularity.
    In particular, it is $F$-liftable.
\end{lem}
\begin{proof}
    The first assertion follows from \cite[Theorem 2.13]{Lee-Nakayama}.
    We note that, for a plt case, we have $B_2=0$ in the notation of \cite[Theorem 2.13]{Lee-Nakayama}, but we can reduce the case $B_2\neq 0$ as in \cite[Proof of Theorems 2.11 and 2.13]{Lee-Nakayama}.
    The last assertion follows from the fact that toric pairs are $F$-liftable \cite[Remark 2.7 (3)]{Kaw4}.
\end{proof}

\begin{proof}[Proof of Theorem \ref{Introthm:F-purity and F-liftability}]
    Since $F$-liftability implies $F$-purity (Corollary \ref{cor:F-lift to F-pure}), it suffices to show that the if direction.
    If $(P\in X,B)$ is canonical, then the assertion follows from Proposition \ref{prop:canonical}. 
    Suppose that $(P\in X,B)$ is klt, but not canonical. In this case, the $F$-purity of $(P\in X)$ is equivalent to the strong $F$-regularity of $(P\in X)$ by \cite[Theorem 1.2]{Hara(two-dim)}, and the assertion follows from \cite[Theorem 2.12 (1)]{Kaw4}.
    Suppose that $(P\in X,B)$ is plt, but not klt. In this case, $(P\in X,B)$ have to be a plt cyclic quotient singularity \cite[3.35 (1)]{Kol13}, and thus it is $F$-liftable by Lemma \ref{lem:cyc quotient are toric}.
\end{proof}

\begin{lem}\label{lem:Dihedral}
    Let $(P\in X,B)$ be a Dihedral quotient singularity \cite[3.40.2]{Kol13}.
    If $(P\in X,B)$ is $F$-pure, then it is $F$-liftable.
\end{lem}
\begin{proof}
    We note that the index of $K_X+B$ is two. This fact can be confirmed by the fact that the discrepancy of the exceptional leaves are $-1/2$.
    Moreover, since $(P\in X,B)$ is $F$-pure, we have $p\neq 2$ by \cite[Theorem 4.5 (2)]{Hara-Watanabe}.
    Let $g\colon (P'\in X',B')\to (X,B)$ be the index one cover \cite[Definition 2.49]{Kol13}.
    Then $(P'\in X',B')$ is lc, and not plt by \cite[Proposition 2.50]{Kol13}. Moreover, all the discrepancy have to be integer since $K_{X'}+B'$ is Cartier. Thus, $(P'\in X',B')$ is a cyclic quotient lc singularity \cite[3.35 (2)]{Kol13}, which is $F$-liftable by Lemma \ref{lem:cyc quotient are toric}.
    Thus, we conclude by Theorem \ref{thm:finite descent}.
\end{proof}

\begin{lem}\label{lem:cone}
    Let $Y$ be an ordinary Abelian variety and $L$ an ample invertible sheaf.
    Then the affine cone
    \[
    X\coloneqq \Spec\,\oplus_{m\geq 0}H^0(Y,L^m)
    \]
    over $Y$ is $F$-liftable.
\end{lem}
\begin{proof}
    In this case, we can show that $X$ and its Frobenius $F_X$ lifts to $W(k)$. 
    We take a canonical lift $\tilde{Y}$, $\tilde{F_Y}$, and $\tilde{L}$ over $W(k)$ of $Y$, $F_Y$, and $L$ (\cite[Appendix, Theorem 1]{MS87}).
    Then $\tilde{X}\coloneqq \Spec\,\oplus_{m\geq 0}H^0(\tilde{Y},\tilde{L}^m)$ is a lift of $X$.
    Since $\tilde{F}^{*}\tilde{L}=\tilde{L}^p$ by \cite[Appendix, Theorem 1 (3)]{MS87}, we have a map $H^0(\tilde{Y},\tilde{L}^m)\to H^0(\tilde{Y},\tilde{L}^{pm})$ for every $m\geq 0$.
    Gluing these maps, we have 
    \[
    \sO_{\tilde{X}}=\oplus_{m\geq 0}H^0(\tilde{Y},\tilde{L}^m)\to \oplus_{m\geq 0}H^0(\tilde{Y},\tilde{L}^{pm})\subset \sO_{\tilde{X}}= \oplus_{m\geq 0}H^0(\tilde{Y},\tilde{L}^m),
    \]
    which gives a lift of the Frobenius map $F_X$.
\end{proof}

\begin{lem}\label{lem:simple elliptic}
    A simple elliptic singularity $(x\in X) $\cite[3.39.1]{Kol13} is $F$-liftable if it is $F$-pure.
\end{lem}
\begin{proof}
    Let $f\colon Y\to X$ be the minimal resolution with $E\coloneqq \Exc(f)$.
    Then $E$ is an ordinary elliptic curve by \cite[Theorem 1.1]{Mehta-Srinivas(surface)}.
    By \cite[Theorem 4.2]{Hirokado(deformation)}, we may assume that $X$ is affine cone $\Spec\,\oplus_{m\geq 0}H^0(E,L^m)$ over $E$ with $L=\sO_E(-E)$.
    Then the assertion follows from Lemma \ref{lem:cone}.
\end{proof}

A cusp singularity is always $F$-pure \cite[Theorem 1.2]{Mehta-Srinivas(surface)}.
We expect that it is also $F$-liftable:

\begin{conj}\label{conj}
    A cusp singularity \cite[3.39.2]{Kol13} is $F$-liftable.
\end{conj}

\begin{lem}\label{lem:lc rational}
    Assume Conjecture \ref{conj}.
    Then a rational lc singularity $(P\in X)$ of type $(2,2,2,2)$ \cite[3.39.3]{Kol13} and of type $(3,3,3)$, $(2,4,4)$, $(2,3,6)$ \cite[3.39.4]{Kol13} 
    is $F$-liftable if it is $F$-pure.
\end{lem}
\begin{proof}
    By \cite[Chapter 3 (3.3.4)]{Flips-abundance}, if $(P\in X)$ is of type $(2,2,2,2)$ (resp.~$(3,3,3)$, $(2,4,4)$, $(2,3,6)$), then 
    the Gorenstein index of $(P\in X)$ is $2$ (resp.~$3,4,6$).
    Thus, if $(P\in X)$ is $F$-pure, then the Gorenstein index is prime to $p$ by \cite[Theorem 1.2]{Hara(two-dim)}.
    Let $(P'\in X')\to (P\in X)$ be the index one cover \cite[Definition 2.49]{Kol13}.
    Then $(P'\in X')$ is Gorenstein and lc, but not klt \cite[Proosition 2.50]{Kol13}.
    Thus $(P'\in X')$ have to be a simple elliptic singularity or a cusp singularity, which is $F$-liftable by Lemma \ref{lem:simple elliptic} and Conjecture \ref{conj}.
    Now, we conclude by Theorem \ref{thm:finite descent}.
\end{proof}

\begin{prop}\label{prop:F-pure to F-lift(non-plt)}
    Assume Conjecture \ref{conj}.
    Let $(P\in X,B)$ be a surface singularity such that $B$ is reduced.
    Suppose that $(P\in X,B)$ is not plt.
    Then $(P\in X,B)$ is $F$-liftable if and only if it is $F$-pure.
\end{prop}
\begin{proof}
    It suffices to prove the if direction.
    In this case, $(P\in X,B)$ is lc by \cite[Proposition 2.2 (b)]{BBKW}.
    Here, we note that sharply $F$-purity is equivalent to $F$-purity if the boundary divisor is reduced.
    Since $(P\in X,B)$ is not plt, it suffices to consider the case of
    \begin{enumerate}
        \item a simple elliptic \cite[3.39.1]{Kol13}, a cusp \cite[3.39.1]{Kol13},
        \item a quotient of a simple elliptic or a cusp \cite[3.39.3 and 3.39.4]{Kol13}, 
        \item a cyclic quotient \cite[3.40.1]{Kol13}, and
        \item a Dihedral quotient \cite[3.40.2]{Kol13}.
    \end{enumerate}
    Then the case (1) (resp.~(2), (3), (4)) follows from Lemma \ref{lem:simple elliptic} and Conjecture \ref{conj} (resp.~Lemma \ref{lem:lc rational}, Lemma \ref{lem:cyc quotient are toric}, Lemma \ref{lem:Dihedral}).
\end{proof}

\subsection{Toward the \texorpdfstring{$F$}--liftability of cusp singularities}\label{subsec:cusp}

In this subsection, we prove the $F$-liftability of some hypersurface cusp singularities in characteristic $2$, which supports Conjecture \ref{conj}.

\begin{prop}\label{prop:cusp}
Suppose that $p=2$.
In the following case, $\Spec\,R/I$ is $F$-liftable.
\begin{enumerate}
\item[\textup{(1)}]
$R=k[x,y,z]$ and $I=(f)$, where 
\begin{equation}
\label{eq:hypsurfcusp}
f=x^a+y^b+z^c+xyz
\end{equation}
for $a,b,c \in \Z$ with $\frac{1}{a} + \frac{1}{b} + \frac{1}{c} < 1$.
\item[\textup{(2)}]
$R=k[x,y,z,w]$ and $I=(f,g)$, where
\begin{align}
\begin{aligned}
\label{eq:completeintersectioncusp}
&f=xy+z^a+w^b, \\ 
&g=zw+x^c+y^d
\end{aligned}
\end{align}
for $a, b, c, d \in \Z$ with $a, b, c, d \geq 2$ and at least one exponent $\geq 3$. 

\end{enumerate}
\end{prop}
\begin{proof}
First, we prove the case (1).
\begin{eqnarray*}
\Delta_1(f) &=& x^ay^b+x^az^c+y^bz^c+xyz(x^a+y^b+z^c)\\
&\equiv& x^ay^b+x^az^c+y^bz^c+ (xyz)^2 \mod (f).
\end{eqnarray*}
Moreover, by using the splitting $\sigma$ defined by
\[
\sigma (F_* r)= u (F_{*}rf)
\]
(we use the same notation as in the proof of Proposition \ref{prop:canonical}),
we have
\begin{eqnarray*}
&\Delta_1 (f) + (-\sigma (F_{\ast} \Delta_1 (f)))^p\\
\equiv&  
\begin{cases}
x^ay^b+x^az^c+y^bz^c+ x^{a-1}y^{b-1}z^{c-1} \mod (f) \textup{ if }a,b,c \textup{ are odd,} \\
x^ay^b+x^az^c+y^bz^c \mod (f) \textup{ otherwise.}
\end{cases}
\end{eqnarray*}
by the direct computation.  
By Lemma \ref{lem:F-lift criterion} (or Theorem \ref{thm:Zdanowicz Critierion}), it suffice to show 
\[
\Delta_1 (f) + (-\sigma (F_{\ast} \Delta_1 (f)))^p \in (f, I_f^{[p]}).
\]
Note that $I_f^{[p]}$ is generated by $f_x^2 =ax^{2a-2} + y^2z^2, f_y^2=by^{2b-2}+x^2z^2, f_z^2=cz^{2c-2}+x^2y^2$.

Note that, when $a,b,c$ are even, the desired result is obvious by Theorem \ref{thm:Zdanowicz Critierion} since $x^ay^b+x^az^c+y^bz^c$ has the square root in $R$.
Therefore, we consider the case where $a,b,c$ are odd firstly.
Since
\begin{eqnarray*}
x^ay^b &=& x^{a-2} y^{b-2} (f_z^2-cz^{2c-2}) \\
&\equiv& x^{a-2} y^{b-2} z^{2c-2} \mod (f, I_f^p),
\end{eqnarray*}
we have 
\begin{eqnarray*}
&&x^ay^b+x^az^c+y^bz^c+ x^{a-1}y^{b-1}z^{c-1} \\
\equiv&& x^{a-2}y^{b-2}z^{c-2} (x^a+y^b+z^c+xyz) \\
\equiv && 0 \mod (f, I_f^{[p]}).
\end{eqnarray*}
Therefore, in this case, we have the desired result.

Next, we consider the case where $a,b$ are odd and $c$ is even.
In this case, by the similar computation as above,
we have
\begin{eqnarray*}
&&x^ay^b+x^az^c+y^bz^c \\
\equiv && x^az^c+y^bz^c \\
\equiv&& x^{a-2}y^{b-2}z^{c-2} (x^a+y^b)  \mod (f, I_f^{[p]})
\end{eqnarray*}
If we have $a,b \geq 4$, then we have
\[
x^{a-2}y^{b-2} \in f_z^p,
\]
so we get the desired result.
Therefore, we may assume $a =3.$
If we have $c \geq 4$, then 
\begin{eqnarray*}
&& x^{2a-2} y^{b-2} z^{c-2} \\
\equiv && x^{2a-4} y^{3b-4} z^{c-4} \\
\equiv && 0  \mod (f, I_f^{[p]}).
\end{eqnarray*}
Similarly, we have $x^{a-2} y^{2b-2} z^{c-2} \equiv 0$, so we get the desired result.
The remaining case is the case where $a=3, c=2$ and $b \geq 7$.
In this case, since we have 
\[
x^{2a-2}y^{b-2}z^{c-2} \in (f_z^p),
\]
it suffices to show that
\[
x^{a-2}y^{2b-2}z^{c-2} = xy^{2b-2} \in (f, I_f^{[p]}).
\]
We can show this as the following:
\begin{eqnarray*}
xy^{2b-2} &\equiv& xy^{b-2} (x^3+z^2+xyz)\\
&\equiv&  xy^{b-2} z^2 \\
&\equiv&  x^{5}y^{b-4}\\
&\equiv& 0 \mod (f, I_f^{[p]}).
\end{eqnarray*}

Finally, we consider the case where $a$ is odd and $b,c$ are even.
In this case, we have
\begin{eqnarray*}
&&x^ay^b+x^az^c+y^bz^c \\
\equiv && y^b z^c \\
\equiv&& x^{2a-2}y^{b-2}z^{c-2}  \mod (f, I_f^{[p]}).
\end{eqnarray*}
Since $\frac{1}{a} + \frac{1}{b} + \frac{1}{c} <1$, at least one of $b,c$ is grater than $3$.
Therefore, the last term is contained in $(f_y^p, f_z^p)$, and it finishes the proof of (1).

Next, we consider the case (2).
To find $g_1, g_2$ (or $h_1, h_2$) as in Remark \ref{rem:F-lift ci}, we actually used the splitting as in Remark \ref{rem:F-lift ci} in a similar way to the proof of (1).
However, to complete the proof, we don't need to state how to find them. Therefore, we omit the description of the splinter in the following, and we only use \cite[Theorem 4.10]{Zda18} here.
In a similar way to the proof of (1), we have
\begin{align*}
\Delta_1(f) & \equiv z^aw^b + x^2y^2  \mod (f) \\
\Delta_1(g) & \equiv x^cy^d + z^2w^2  \mod (g).
\end{align*}
Therefore, by \cite[Theorem 4.10]{Zda18} (see Remark \ref{rem:F-lift ci}), it suffices to show that there exist $g_1, g_2, h_1, h_2, h_3, h_4 \in R$ such that 
\begin{eqnarray*}
&&\left( \begin{array}{c}
z^aw^b \\
x^cy^d
\end{array}
\right)
+
\left(
\begin{array}{c}
g_1^p\\
g_2^p
\end{array}
\right)\\
=&&
h_1 
\left( \begin{array}{c}
y^2 \\
cx^{2c-2}
\end{array}
\right)
+h_2
\left( \begin{array}{c}
x^2 \\
dy^{2d-2}
\end{array}
\right)
+h_3
\left( \begin{array}{c}
az^{2a-2} \\
w^2
\end{array}
\right)
+ h_4
\left( \begin{array}{c}
bw^{2b-2} \\
z^2
\end{array}
\right) \mod (f,g)
\end{eqnarray*}

We divide into cases according to the parity of $(a,b,c,d)$.
If $(a,b,c,d) \equiv (0,0,0,0) \mod 2$, it is obvious.
First, consider the case where $(a,b,c,d) \equiv (1,1,1,1)$.
In this case, we have
\[
z^a w^b \equiv z^{a-2} w^{b-2} (x^{2c}+y^{2d}) \mod (g)
\]
and
\[
z^{a-2} w^{b-2} y^{2d-2}
\left( \begin{array}{c}
y^2 \\
cx^{2c-2}
\end{array}
\right)
+ z^{a-2} w^{b-2} x^{2c-2}
\left( \begin{array}{c}
x^2 \\
dy^{2d-2}
\end{array}
\right)
=
\left( \begin{array}{c}
z^{a-2} w^{b-2} (x^{2c} +y^{2d}) \\
0
\end{array}
\right).
\]
By combining with a  similar computation for $x^{c}y^{d}$, we have the desired result.

Next, we consider the case where $(a,b,c,d) \equiv (1,1,1,0)$.
In this case, we have
\[
\left( \begin{array}{c}
z^aw^b +z^{a-1}w^{b-1}y^d \\
x^cy^d
\end{array}
\right)
\equiv
\left( \begin{array}{c}
z^{a-1}w^{b-1}(x^c+y^d+y^d) \\
x^cy^d
\end{array}
\right)
=
\left( \begin{array}{c}
z^{a-1}w^{b-1}x^c\\
x^cy^d
\end{array}
\right) \mod (g),
\]
and the last vector can be modified to 
\[
\left( \begin{array}{c}
0\\
x^cy^d
\end{array}
\right)
\]
by using $^{t}\!(f_y^p, g_y^p) = ^{t}\!(x^2, 0)$.
By the same computation for $x^cy^d$ as in the case where $(a,b,c,d) \equiv (1,1,1,1)$, we have the desired result.

Next, we consider the case where $(a,b,c,d) \equiv (1,0,1,0)$.
In this case, we have
\[
 \left( \begin{array}{c}
z^aw^b \\
x^cy^d
\end{array}
\right)
\equiv
\left( \begin{array}{c}
z^{a-1}w^{b-1}(x^c+y^d+y^d) \\
x^cy^d
\end{array}
\right)
\equiv
\left( \begin{array}{c}
z^{a-1}w^{b-1}(x^{c}+y^{d}) \\
x^cy^d
\end{array}
\right) \mod (g).
\]
By using $^{t}\!(f_y^p, g_y^p) = ^{t}\!(x^2, 0)$ and $^{t}\!(f_x^p, g_x^p) = ^{t}\!(y^2, x^{2c-2})$, the last term can be modified to
\[
\left( \begin{array}{c}
0 \\
x^cy^d + z^{a-1}w^{b-1}x^{2c-2}y^{d-2}
\end{array}
\right),
\]
which can be modified further to
\[
\left( \begin{array}{c}
0 \\
x^cy^d
\end{array}
\right)
\]
by using $^{t}\!(f_w^p, g_w^p) = ^{t}\!(0, z^2)$ since $a \geq 3.$
By the same computation for $x^cy^d$, we have the desired result.

Finally, we consider the case where $(a,b,c,d) \equiv (1,0,0,0)$.
In this case,
the same method as in the case where $(a,b,c,d \equiv (1,0,1,0)$ works.
It finishes the proof.
\end{proof}

\begin{rem}\,
\label{rem:cuspsingularity}
\begin{enumerate}
\item
In characteristic zero, the equations in Proposition \ref{prop:cusp} give all the complete intersection cusp singularities by \cite[Theorem 4]{Kar}.

\item 
By the direct computation of blow-ups, we can confirm that the dual graph of minimal resolution of a singularity defined by equations (\ref{eq:hypsurfcusp}) and (\ref{eq:completeintersectioncusp}) in characteristic $p >0$ is the same as in characteristic $0$.
In particular, they have cusp singularities.
\item
By using a computer algebra system, we can confirm the $F$-liftability of these equations in characteristic $p$ other than $2$.
For example, we can check that the equation (\ref{eq:hypsurfcusp}) is $F$-liftable if $p \leq 19$ and $a,b,c \leq 30$.
\end{enumerate}
\end{rem}

\section{Proof of Theorem \ref{Introthm:LET}}
In this section, we apply Theorem \ref{Introthm:F-purity and F-liftability} to prove Theorem \ref{Introthm:LET}.

\begin{proof}[Proof of Theorem \ref{Introthm:LET}]
First, we reduced the case where $(X,B)$ is dlt.
Suppose that $(X,B)$ is not dlt.
Let $h\colon (W, B_W\coloneqq h^{-1}_{*}B+\Exc(h))\to (X,B)$ be a dlt blow-up  (see \cite[Definition 4.3 and Lemma 4.4]{Kaw3} for example).
\begin{cln}\label{cln:dlt blow-up}
    The Cartier index of $K_W+B_W$ is not divisible by $p$.
\end{cln}
\begin{proof}
    Since $K_W+B_W=h^{*}(K_X+B)$, it follows that $(W, B_W)$ is $F$-pure by \cite[Lemma 2.7]{GT16}.
    From now, we divide the cases according to (7.8.1)--(7.8.7) in \cite[7.B]{Gra}.
    \begin{itemize}
        \item \textbf{(7.8.1), (7.8.2), and (7.8.5) cases.}\,\,In this case, $(W,B_W)$ is a log resolution.
        \item \textbf{(7.8.3) and (7.8.6) cases.}\,\,
        Since $(W, B_W)$ is $F$-pure, we have $p\neq 2$ by \cite[Theorem 1.2 (3)-(vi) and (4)]{Hara(two-dim)} and \cite[Theorem 4.5 (2)]{Hara-Watanabe}. Then every $\Z$-divisor on $W$ has Cartier index at most two since $W$ is obtained by contracting $(-2)$-curves of a smooth surface.
        \item \textbf{(7.8.4) case.}\,\,In this case, $B=0$, and the singularity is a rational lc singularity of type $(3,3,3)$, $(2,3,6)$, or $(2,4,4)$. 
        Since $W$ is $F$-pure, the Cartier index of $K_W$ is not divisible by $p$ by \cite[Theorem 1.2 (3)-(iv) and (3)-(v)]{Hara(two-dim)}.
    \end{itemize}
    Now, we have covered all the cases (recall that we assumed that $(X,B)$ is not dlt), and we obtain the claim.
\end{proof}

Combining Claim \ref{cln:dlt blow-up} with the claim in \cite[Proof of Proposition 4.8]{Kaw3}, the assertion is reduced to $(W,B_W)$, and thus we may assume that $(X,B)$ is dlt. 

If $p\neq 5$, or $p=5$ and there does not exist an RDP $(P\in X)$ of type $E_8^1$,
then $(X,B)$ is locally $F$-liftable by Theorem \ref{Introthm:F-purity and F-liftability}, and thus the assertion follows from Theorem \ref{thm:F-lift}.

Suppose that $p=5$ and $X$ has an RDP $(P\in X)$ of type $E_8^1$.
Since the assertion is local on $(X,B)$, we may assume that $B=0$ and $P\in X$ is the only singular point of $X$.
By shrinking $X$ further if necessary, we can assume that the class group $\mathrm{Cl}(X)=0$ by \cite[Section 24]{Lipman} (see also \cite[Table 2]{LMM2}), and in particular $D=0$.
In this case, we can conclude the assertion from \cite[Theorem 1.1 (ii)]{Hirokado} (see also Theorem \ref{thm:RDPs}).
\end{proof}

\section{Proof of Theorem \ref{Introthm:BSV}}
In this section, we prove Theorem \ref{Introthm:BSV}.
Applying Theorem \ref{Introthm:LET}, we first prove Akizuki-Nakano vanishing for globally $F$-split surface pairs. Then using the minimal model program, we deduce Bogomolov-Sommese vanishing for globally $F$-split surface pairs.

\begin{lem}\label{lem:MFS}
Let $(X,B)$ be a globally $F$-split projective surface pair over an algebraically closed field of positive characteristic such that $B$ is reduced.
Let $D$ be a $\Z$-divisor on $X$ and $\phi\colon X\to Z$ be a projective surjective morphism to a smooth curve $Z$
such that 
\begin{enumerate}
 \item[\textup{(1)}] $\phi_{*}\sO_X=\sO_Z$
 \item[\textup{(2)}] $D\cdot F>0$ for a general fiber $F$ of $\phi$, and 
 \item[\textup{(3)}] $-(K_X+B)\cdot F\geq 0$ for a general fiber $F$.
\end{enumerate}
Then 
\[
\phi_{*}\Omega_X^{[1]}(\log\,B)(-D)=0.
\]
\end{lem}
\begin{proof}
Since $\phi_{*}\Omega_X^{[1]}(\log\,B)(-D)$ is torsion-free, it suffices to show that the rank of this sheaf is zero.
Thus the assertion is local on $Z$ and we can shrink $Z$ if necessary.
In particular, we may assume that $Z$ is affine, $X$ is smooth, and $B$ is Cartier.
By using \cite[Proposition 5.11]{Eji19} for $\phi$ and the structure morphisms, it follows that $(\phi, B)$ is $F$-split. 
Then \cite[Proposition 5.7]{Eji19} shows that a general fiber $F$ is normal and $(F, B|_{F})$ is globally $F$-split.
Since $F$ is globally $F$-split, $(1-p)K_F$ is linearly equivalent to an effective divisor. Thus, $F$ is the projective line or an elliptic curve.
Since $(F, B|_{F})$ is globally $F$-split and $\deg(-(K_F+B|_{F}))\geq 0$, one of the following holds.
\begin{itemize}
    \item $F\cong\PP_k^1$ and $B|_F$ is zero, a point, or two distinct points. 
    \item $F$ is an elliptic curve and $B|_F=0$.
\end{itemize}
In particular, every $f$-horizontal prime divisor $C \subseteq \Supp B$ is generically \'etale over $Z$ in each case, and by shrinking $Z$, we can assume that $(X,B)$ is simple normal crossing over $Z$.
Now, the proof is essentially same as \cite[Lemma 4.11]{Kaw3}, but we include the proof for the completeness. 

Since $Z$ is affine, we have
\[\phi_{*}\Omega_X^{[1]}(\log\,B)(-D)=H^0(X, \Omega_X^{[1]}(\log\,B)(-D)).\]
Suppose by contradiction that $H^0(X, \Omega_X^{[1]}(\log\,B)(-D))\neq 0$.
Then there exists an injective $\sO_X$-module homomorphism
\[
s\colon \sO_X(D)\hookrightarrow \Omega_X^{[1]}(\log\,B).
\]
Since $(X,B)$ is simple normal crossing over $Z$, we have the following horizontal exact sequence:
\begin{equation*}
\xymatrix{ & & \sO_{X}(D) \ar@{.>}[ld] \ar[d]^{s} \ar[rd]^{t} &\\
                0\ar[r] &\sO_X(f^{*}K_Z)\ar[r]   & \Omega^{[1]}_X(\log B) \ar[r]^{\rho}  & \Omega^{[1]}_{X/Z}(\log B) \ar[r] & 0.}
\end{equation*}
We refer to the proof of \cite[Lemma 4.11]{Kaw3} for the construction of the exact sequence. 
Set $t \coloneqq \rho \circ s$. 
Suppose that $t$ is nonzero. 
Then, by restricting $t$ to $F$, we have an injective $\sO_F$-module homomorphism \[t|_{F}\colon \sO_F(D|_F) \hookrightarrow 
\Omega^1_{F}(\log (B|_F))=\sO_F(K_F+B|_{F}),\] where the injectivity holds since $F$ is chosen to be general. 
Since $-(K_{X}+B)\cdot F \geq 0$,
we have 
\[
0<\deg (D|_F)\leq\deg(K_F+B|_F)\leq 0,
\]
a contradiction.
Thus $t$ is zero and the homomorphism $s$ induces an injection  $\sO_X(D)\hookrightarrow \sO_X(\pi^{*}K_Z)$.
By taking the restriction to $F$, we get
\[
0<\deg (D|_F)\leq\deg(f^*K_Z|_F)=0,
\]
which is again a contradiction. 
Therefore, we obtain the required vanishing.
\end{proof}

 We recall that a $\Z$-divisor $D$ on a normal projective surface $X$ is said to be \textit{nef} if $D\cdot C\geq0$ for every curve $C$ on $X$.

\begin{prop}[Akizuki-Nakano vanishing for globally $F$-split surface pairs]\label{prop:ANV}
Let $(X,B)$ be a globally $F$-split projective surface pair over an algebraically closed field of positive characteristic such that $B$ is reduced.
Let $D$ be a nef and big $\Z$-divisor on $X$.
Then 
\[
H^j(X, \Omega_X^{[i]}(\log\,B)(-D))=0
\]
for $i+j<2.$
\end{prop}
\begin{proof}
If $(i,j)=(0,0)$, then the assertion follows from the bigness of $D$.
Thus, we may assume that $(i,j)=(0,1)$ or $(1,0)$.
Let $f\colon Y\to X$ be the minimal log resolution. Then $f^{*}(K_X+B)-K_Y$ is effective as follows:

Let $\pi\colon Z\to X$ be the minimal resolution. Then $\pi^{*}(K_X+B)-K_Z$ is always effective.
From the the classification of lc surface singularities (\cite[3.39, 3.40]{Kol13}), we can confirm that
every non-log smooth point of $(X', \pi_{*}B+\Exc(\pi))$ are nodal points of irreducible components
in $\pi^{*}(K_X+B)-K_Z$ with coefficient one.
Then taking the blow-up along the nodal points, we obtain the minimal log resolution $f\colon Y\to X$ and $f^{*}(K_X+B)-K_Y$ is effective. 

Since $B_Y$ is effective, the pair $(Y,B_Y)$ is globally $F$-split by \cite[Lemma 2.7]{GT16}, and in particular, $Y$ is globally $F$-split.
Then $(Y, f_{*}^{-1}B+\Exc(f))$ lifts to $W_2(k)$ by \cite[Lemma 1.10]{BBKW}.
By Theorem \ref{Introthm:LET}, we have
\[
H^0(X, \Omega_X^{[1]}(\log\,B)(-D))=H^0(Y, \Omega_Y^{1}(\log\,f_{*}^{-1}B+\Exc(f))(-f^{*}D)).
\]
By the Leray spectral sequence, 
we also have 
\[
H^1(X, \sO_X(-D))\hookrightarrow H^1(Y, \sO_Y(-f^{*}D)).
\]
Now, the assertion follows from Akizuki-Nakano vanishing for $W_2(k)$-liftable log smooth surface pairs \cite[Theorem 2.11]{Kaw3}.
\end{proof}

Now, we prove Theorem \ref{Introthm:BSV}.

\begin{proof}[Proof of Theorem \ref{Introthm:BSV}]
Since the assertion is obvious when $i=0$ or $2$, we may assume that $i=1$.
Since $(X, \lfloor B\rfloor)$ is globally $F$-split, replacing $B$ with $\lfloor B\rfloor$, we may assume that $B$ is reduced.
By \cite[Proposition 2.2]{BBKW}, the pair $(X, B)$ is lc and $-(K_X+B)$ is $\Q$-effective, i.e., there exists $m\in\Z_{>0}$ such that $-m(K_X+B)$ is linearly equivalent to an effective divisor.
Here, we recall that global sharply $F$-splitting of $(X,B)$ is equivalent to global $F$-splitting of $(X,B)$ since $B$ is reduced.

Let $f\colon Y\to X$ be a dlt blow-up with $B_Y=f^{-1}_{*}B+\Exc(f)$.
Since $f$ is a dlt blow-up and $(X,B)$ is lc, we have $(Y,B_Y)$ is dlt and $K_Y+B_Y=f^{*}(K_X+B)$ (see \cite[Definition 4.3 and Lemma 4.4]{Kaw3} for example).
In particular, $-(K_Y+B_Y)$ is $\Q$-effective, and thus $(Y,B_Y)$ is globally $F$-split by \cite[Lemma 2.7]{GT16}.

By Theorem \ref{Introthm:LET}, 
we have 
\[
H^0(X, \Omega_X^{[1]}(\log\,B)(-D))=H^0(Y, \Omega_Y^{[1]}(\log\,B_Y)(-f^{*}D)).
\]
Thus,by replacing $(X,B)$ with its dlt model $(Y, B_Y)$, we may assume that $(X,B)$ is dlt. 

We divide the proof into three cases. We note that if $K_X+B$ is pseudo-effective, then $K_X+B\equiv 0$ since $-(K_X+B)$ is $\Q$-effective.

\textbf{The case where $K_X+B$ is not pseudo-effective.}
For any birational projective morphism $\phi\colon X\to X'$ of normal varieties, the pair $(X',B'\coloneqq \phi_{*}B)$ is globally $F$-split by \cite[Lemma 2.3]{BBKW} and the desired vanishing can be reduced to $(X',B')$ by \cite[Lemma 4.1]{Kaw3}.

In particular, by running a $(K_X+B)$-MMP, we may assume that $(X,B)$ has a $(K_X+B)$-Mori fiber space structure $f\colon X\to Z$.
If $\dim\,Z=1$, then the assertion follows from Lemma \ref{lem:MFS}.
Suppose that $\dim\,Z=0$. In this case, the Picard number $\rho(X)=1$ and thus $D$ is ample. Here we use that $X$ is $\Q$-factorial since $(X,B)$ is dlt. Then the assertion follows from Proposition \ref{prop:ANV}.

\textbf{The case where $K_X+B\equiv 0$ and $B\neq 0$.}
In this case, $K_X$ is not pseudo-effective.
As in the previous case, by running a $K_X$-MMP, we may assume that $(X,B)$ has a $K_X$-Mori fiber space structure by \cite[Lemma 2.3]{BBKW} and \cite[Lemma 4.1]{Kaw3}.
Then an argument of the previous case works. That is, if the base of the Mori fiber space is a curve, then we apply Lemma \ref{lem:MFS}, and if the base is a point, we apply Proposition \ref{prop:ANV}.

\textbf{The case where  $K_X+B\equiv 0$ and $B=0$.}
In this case, $X$ is a klt surface with $K_X\equiv 0$. 
Then, considering the Zariski decomposition, we can reduce to the case where $D$ is nef and big in a way similar to \cite[Proof of Theorem 1.2]{Kaw3} as follows:

Let $D\equiv P+N$ be the Zariski decomposition. 
(see \cite[Theorem 3.1]{Eno} for the Zariski decomposition on normal surfaces).
We take a small rational number $0<\epsilon \ll 1$ such that $(X, \epsilon N)$ is klt.
Since $K_X$ is torsion by the abundance theorem (\cite[Theorem 1.2]{Tan12}) and $N$ is negative definite, it follows that $\kappa(K_X+\epsilon N)=\kappa(X, N)=0$.
We run a $(K_X+\epsilon N)$-MMP to obtain a birational contraction $\phi\colon X\to X'$ to a $(K_X+\epsilon N)$-minimal model $X'$.
Then $K_{X'}=\phi_{*}K_X\equiv0$, and in particular, $X'$ is klt Calabi-Yau.
Moreover, $\phi_{*}\epsilon N\equiv K_{X'}+\phi_{*}\epsilon N\equiv 0$, and hence $\phi_{*}D\equiv\phi_{*}P$ is nef and big. 
Finally, by \cite[Lemma 2.3]{BBKW} and \cite[Lemma 4.1]{Kaw3} again
we can replace $X$ with $X'$, and can assume that $D$ is nef and big.

Now, the assertion follows from Proposition \ref{prop:ANV}.
\end{proof}

\begin{rem}
    A partial result on Bogomolov-Sommese vanishing on globally $F$-split threefolds can be found in \cite{Kaw1}.
\end{rem}

\section{\texorpdfstring{$F$}--singularities and the extendability of differential forms for RDPs}
In this section, we summarize the $F$-purity, $F$-regularity, $F$-liftability, and the regular and logarithmic extendability of differential forms for RDPs. 
The $F$-purity and $F$-regularity of RDPs are well-known (see \cite[Proposition 2.1]{Fedder} and \cite[Theorem 1.2]{Hara(two-dim)} for example).
We have checked the $F$-liftability of RDPs in Section \ref{subsec:F-lift for RDPs}.
The regular extendability of differential forms for RDPs was completely investigated by Hirokado \cite[Theorem 1.1 (iii)]{Hirokado}, and we can also clarify the logarithmic extendability of differential forms by combining his results.

We remark that the logarithmic extendability of differential forms in this section is not enough to show Theorems \ref{Introthm:LET} and \ref{Introthm:BSV} as explained in Remark \ref{rem:intro}. 

\begin{conv}
Let $X$ be an affine surface that has an RDP
$P\in X$ and is smooth outside $P$.
We denote the minimal resolution of $X$ by $f\colon Y\to X$.
Then $\omega_X\cong \sO_X$ and $\omega_Y\cong \sO_Y$ hold, and
in particular, we obtain $T_X\cong \Omega_X^{[1]}$, $T_Y\cong \Omega^1_Y$, and $T_Y(-\log\,E)(E)\cong \Omega^1_Y(\log\,E)$.
\end{conv}

\begin{lem}[\textup{cf.~\cite[Theorem 1.1]{Hirokado} (iii)}]\label{lem:RET}
    We have an exact sequence
    \[
    0\to f_{*}\Omega_Y\to \Omega_X^{[1]} \to H^1_{E}(T_Y(-\log\,E))\to 0.
    \]
\end{lem}
\begin{proof}
   We note that 
    \begin{align*}
    &f_{*}\Omega^1_Y=H^0(Y, \Omega^1_Y)\\
    &\Omega_X^{[1]}=H^0(Y\setminus E, \Omega^1_Y)
    \end{align*}
    since $X$ is affine.
   By the local cohomology exact sequence, we have an injective map
   \[
   \Omega_X^{[1]}/f_{*}\Omega_Y \hookrightarrow H^1_{E}(T_Y(-\log\,E)).
   \]
     By the formal duality (\cite[Chapter 3, Theorem 3.3]{Hartshorne70}), we have $H^1_{E}(T_Y(-\log\,E))\cong R^1f_{*}\Omega^1_Y(\log\,E)$ since $\omega_Y\cong \sO_Y$.
     In particular, both $\Omega_X^{[1]}/f_{*}\Omega^1_Y$ and $H^1_{E}(T_Y(-\log\,E))$ are supported on $P$ and finite dimensional over $k$.
     Thus, it suffices to show that $\dim_{k} \Omega_X^{[1]}/f_{*}\Omega_Y=\dim_{k} H^1_{E}(T_Y(-\log\,E))$.
     Let $\tau\colon X'\coloneqq \Spec\,\sO_{X,P}\to X$ the natural map.
     We denote by $Y',E',f'$ the base change of $Y,E,f$ by $\tau$, respectively.
     Since $\tau$ is flat, we have 
     \[
     \tau^{*}(\Omega_X^{[1]}/f_{*}\Omega^1_Y)=\Omega_{X'}^{[1]}/(f')_{*}\Omega^1_{Y'}.
     \]
     We also have
     \begin{align*}
     \tau^{*}H^1_{E}(T_{Y}(-\log\,E)) 
     &= \tau^{*}\mathcal{H}^1R\Gamma_{P}(Rf_{*} (T_{Y} (-\log\,E)))  \\
     &= \mathcal{H}^1R\Gamma_{P}(\tau^{*}Rf_{*} (T_{Y} (-\log\,E))) \\
     &= \mathcal{H}^1R\Gamma_{P} (R(f')_{*} (T_{Y'}(-\log\,E'))) \\
     &= H^{1}_{E'} (T_{Y'} (-\log\,E'))
     \end{align*}
     Therefore, we can replace $X$ with $X'$, and we obtain the assertion by \cite[Theorem 1.1 (iii)]{Hirokado}.
\end{proof}

\begin{lem}\label{lem:LET for RDPs}
    We have the exact sequence
    \[
    0\to f_{*}\Omega^1_Y(\log\,E)\to \Omega_X^{[1]} \to H^1_{E}(\Omega^1_Y(\log\,E))\to 0.
    \]
\end{lem}
\begin{proof}
    As in Lemma \ref{lem:RET}, it suffices to show that 
    \[
    \dim_{k} H^0(Y\setminus E, \Omega^1_Y(\log\,E))/H^0(Y, \Omega^1_Y(\log\,E))=\dim_{k} H^1_{E}(\Omega^1_Y(\log\,E)).
    \]
    By Lemma \ref{lem:RET},
    we have 
    \begin{align*}
    \dim_{k} H^0(Y\setminus E, \Omega^1_Y(\log\,E))/H^0(Y, \Omega^1_Y)=&\dim_{k} H^0(Y\setminus E, \Omega^1_Y)/H^0(Y, \Omega^1_Y)\\
    =&\dim_{k} H^1_{E}(T_Y(-\log\,E))\\
    =&\dim_{k}\,H^1(Y, \Omega^1_Y(\log\,E)),
    \end{align*}
    where the last equality is the formal duality. 
   Then we have
    \begin{align*}
    &\dim_{k} H^0(Y\setminus E, \Omega^1_Y(\log\,E))/H^0(Y, \Omega^1_Y(\log\,E))\\
    =&\dim_{k} H^0(Y\setminus E, \Omega^1_Y(\log\,E))/H^0(Y, \Omega^1_Y)-\dim_{k} H^0(Y, \Omega^1_Y(\log\,E))/H^0(Y, \Omega^1_Y)\\
    =&\dim_{k}\,H^1(Y, \Omega^1_Y(\log\,E))-\dim_{k} H^0(Y, \Omega^1_Y(\log\,E))/H^0(Y, \Omega^1_Y).
    \end{align*}
    By the exact sequence
    \[
    0\to \Omega^1_Y \to \Omega^1_Y(\log\,E)\to \bigoplus_{i=1}^{d}\sO_{E_i}\to 0,
    \]
    we obtain
    \[
    \dim_{k} H^0(Y, \Omega^1_Y(\log\,E))/H^0(Y, \Omega^1_Y)=d-\dim_k H^1(Y, \Omega^1_Y)+\dim\,H^1(Y, \Omega^1_Y(\log\,E)).
    \]
    By \cite[Theorem 2.4]{Hirokado} and $\Omega^1_Y\cong T_Y$, we have
    \[
    \dim_k H^1(Y, \Omega^1_Y)=d+\dim_{k} H^1(T_Y(-\log\,E))=d+\dim_{k} H^1_{E}(Y, \Omega^1_Y(\log\,E)).
    \]
    Therefore, we obtain
    \begin{align*}
    &\dim_{k} H^0(Y\setminus E, \Omega^1_Y(\log\,E))/H^0(Y, \Omega^1_Y(\log\,E))\\
    =&\dim_{k} H^1(Y, \Omega^1_Y(\log\,E))-(d-\dim_k H^1(Y, \Omega^1_Y)+\dim\,H^1(Y, \Omega^1_Y(\log\,E)))\\
    =&\dim_{k}\,H^1(Y, \Omega^1_Y(\log\,E))-(d-(d+\dim_{k} H^1_{E}(Y, \Omega^1_Y(\log\,E)))+\dim\,H^1(Y, \Omega^1_Y(\log\,E)))\\
    =&\dim_{k} H^1_{E}(Y, \Omega^1_Y(\log\,E)),
    \end{align*}
    as desired.
\end{proof}

We recall the regular and logarithmic extendability of differential forms:

\begin{defn}[Regular and logarithmic extension theorem]\label{def:log ext thm}
Let $X$ be a normal variety and $B$ a reduced divisor on $X$. 
\begin{enumerate}
\item[\textup{(1)}] We say that $X$ satisfies the \textit{regular extension theorem} if, for any proper birational morphism $f\colon Y\to X$ from a normal variety $Y$, the restriction map
\[
f_{*}\Omega^{[i]}_Y\hookrightarrow \Omega^{[i]}_X
\]
is an isomorphism for all $i\geq 0$.
\item[\textup{(2)}] We say that $(X,B)$ satisfies the \textit{logarithmic extension theorem} if, for any proper birational morphism $f\colon Y\to X$ from a normal variety $Y$, the restriction map
\[
f_{*}\Omega^{[i]}_Y(\log f^{-1}_{*}B+E)\hookrightarrow \Omega^{[i]}_X(\log B)
\]
is an isomorphism for all $i\geq 0$, where $E$ is the reduced $f$-exceptional divisor.
\end{enumerate}
If a resolution (resp.~log resolution) exists, then one can verify the regular (resp.~logarithmic) extendability for one resolution (resp.~log resolution) (see \cite[Lemma 2.13]{GKK}).
\end{defn}

\begin{rem}
    Let $(P\in X)$ be a singularity of a pair of a normal variety $X$.
    We say $(P\in X)$ satisfies the regular extension theorem if 
    \[
    f_{*}\Omega^{[i]}_Y\hookrightarrow \Omega^{[i]}_X
    \]
    is surjective at $P$ for some representative $(P\in X)$ and some resolution $f\colon Y\to X$. 
    We define the logarithmic extendability for a singularity $(P\in X,B)$ similarly.
\end{rem}

\begin{thm}\label{thm:RDPs}
Let $(P\in X)$ be an RDP.
For each type of $P$, the $F$-purity, $F$-regularity, $F$-liftability, the regular and logarithmic extendability are shown as in Table \ref{table:RDPs},
where RET and LET mean the regular and logarithmic extension theorem, respectively. 
 \begin{tiny}
\begin{table}[ht]
\caption{$F$-purity, $F$-regularity, $F$-liftability, and the extendability for RDPs}\label{table:RDPs}
\centering
\begin{tabular}{|c|c|l|c|c|c|c|c|}
\hline
$p$ & type & equation & $F$-pure? & $F$-regular? & $F$-liftable? & RET? & LET? \\
\hline
\hline
$>0$ & $A_{n}$ $(p\nmid(n+1))$ &$z^{n+1} +xy$ & Y& Y & Y & Y & Y\\
 \hline
$>0$ & $A_{n}$ $(p|(n+1))$  &$z^{n+1} +xy $ & Y &Y & Y & N & Y\\
 \hline
$>2$ & $D_{n}$ & $z^2 +x^{2}y+y^{n+1}$ & Y &Y& Y & Y & Y\\
\hline
$>3$ & $E_{6}$ & $z^2+x^3+y^4$ & Y &Y & Y & Y & Y\\
\hline
$>3$ & $E_{7}$ & $z^2+x^3+xy^3$ & Y &Y & Y & Y & Y\\
\hline
$>5$ & $E_{8}$ & $z^2+x^3+y^5$  & Y &Y & Y & Y & Y\\
 \hline
$2$ & $D_{2n}^{0}(n\geq 2)$ &$z^2 +x^{2}y+xy^n$ & N &N & N & N & N\\
 \hline
$2$ & $D_{2n}^{r}$ $(r=1, \ldots, n-2)$ &$z^2 +x^{2}y+xy^n+xy^{n-r}z$ & N &N& N & N & N \\
\hline
$2$ & $D_{2n}^{n-1}(n\geq 2)$ &$z^2 +x^{2}y+xy^n+xyz$ & Y&N & Y & N & Y\\
\hline
$2$ & $D_{2n+1}^{0}(n\geq 2)$ &$z^2+x^2y+y^nz$ & N&N &  N & N & N\\
\hline
$2$ & $D_{2n+1}^{r}$ $(r=1,\ldots, n-2)$ &$z^2+x^2y+y^nz+xy^{n-r}z$ & N&N  &  N & N & N\\
\hline
$2$ & $D_{2n+1}^{n-1}(n\geq 2)$ &$z^2+x^2y+y^nz+xy^{n-r}z$ & Y&N & Y & N & Y \\
\hline
$2$ & $E_{6}^{0}$ &$z^2+x^3+y^2z $& N&N & N & N & N\\
\hline
$2$ & $E_{6}^{1}$ &$z^2+x^3+y^2z+xyz$ & Y&N & Y & Y & Y\\
\hline
$2$ & $E_{7}^{0}$ &$z^2+x^3+xy^3$ & N&N & N & N & N\\
\hline
$2$ & $E_{7}^{1}$ &$z^2+x^3+xy^3+x^2yz$& N&N & N & N & N\\
\hline
$2$ & $E_{7}^{2}$ &$z^2+x^3+xy^3+y^3z$ & N&N & N & N & N\\
\hline
$2$ & $E_{7}^{3}$ &$z^2+x^3+xy^3+xyz$& Y&N & Y & Y & Y\\
\hline
$2$ & $E_{8}^{0}$ &$z^2+x^3+y^5$ & N&N & N & N & N\\
\hline
$2$ & $E_{8}^{1}$ &$z^2+x^3+y^5+xy^3z$ & N&N & N & N & N\\
\hline
$2$ & $E_{8}^{2}$ &$z^2+x^3+y^5+xy^2z$ & N&N & N & N & N\\
\hline
$2$ & $E_{8}^{3}$ &$z^2+x^3+y^5+y^3z$ & N&N & N & N & N\\
\hline
$2$ & $E_{8}^{4}$ &$z^2+x^3+y^5+xyz$& Y&N & Y & Y& Y\\
\hline
$3$ & $E_{6}^{0}$ & $z^2+x^3+y^4$ & N&N & N & N & N\\ 
\hline
$3$ & $E_{6}^{1}$ &$z^2+x^3+y^4+x^2y^2$ & Y&N & Y & N & Y\\ 
\hline
$3$ & $E_{7}^{0}$ & $z^2+x^3+xy^3$ & N&N & N & N & N\\ 
\hline
$3$ & $E_{7}^{1}$ &$z^2+x^3+xy^3+x^2y^2$ & Y&N & Y & Y & Y\\ 
\hline
$3$ & $E_{8}^{0}$ & $z^2+x^3+y^5$ & N&N & N & N& N \\ 
\hline
$3$ & $E_{8}^{1}$ & $z^2+x^3+y^5+x^2y^3$ &N&N  & N & N & N\\ 
\hline
$3$ & $E_{8}^{2}$ & $z^2+x^3+y^5+x^2y^2$ &Y&N & Y & Y & Y\\ 
\hline
$5$ & $E_{8}^{0}$ &$z^2+x^3+y^5$  & N&N & N & N & N\\ 
\hline
$5$ & $E_{8}^{1}$ &$z^2+x^3+y^5+xy^4$ & Y&N & N & Y & Y\\ 
\hline
\end{tabular}
\end{table}
\end{tiny}
\end{thm}
\begin{proof}
    The $F$-purity of each RDP can be checked by Fedder's critierion \cite[Propostion 2.1]{Fedder}.
    For the $F$-regularity, we refer to \cite[Theorem 1.1]{Hara(two-dim)}. 
    We can determine $F$-liftability of RDPs by Theorem \ref{Introthm:F-purity and F-liftability}.

    By Lemmas \ref{lem:RET} and \ref{lem:LET for RDPs}, it suffices to see the dimension of $H^1_{E}(T_Y(-\log\,E))$ and $H^1_{E}(\Omega^1_Y(\log\,E))$ for the regular and logarithmic extendability, respectively.
    As in the proof of Lemma \ref{lem:RET}, we can replace $X$ with $\Spec\,\sO_{X,P}$.  
    Then we can use \cite[Theorem 1.1 (ii), (iii)]{Hirokado}, which determined the dimension of $H^1_{E}(T_Y(-\log\,E))$ and $H^1_{E}(\Omega^1_Y(\log\,E))$ for all the RDPs. Note that $S_Y$ in \cite{Hirokado} is equal to $T_Y(-\log\,E)$ in our notation and $T_Y(\log\,E)(E)\cong\Omega^1_Y(\log\,E)$ since $\omega_Y=\sO_Y$.
    
    Summarizing, we obtain Table \ref{table:RDPs}.
\end{proof}

\begin{rem}
    The failure of the regular extension theorem is also proved in \cite[Proposition 1.5]{Langer19} and \cite[Example 10.2]{Gra} in different ways.
    On the other hand, it holds for tame quotients \cite[Theorem C]{Kaw4}.
\end{rem}

\section*{Acknowledgements}
The author wishes to express his gratitude to Adrian Langer and Shou Yoshikawa for valuable discussion.
The first author was supported by JSPS KAKENHI Grant number JP22J00272.
The second author was supported by JSPS KAKENHI Grant number JP22J00962.


\end{document}